\documentclass{article}
\usepackage[tbtags]{amsmath}
\usepackage{amsfonts,amssymb,amsopn,amscd,amsthm,xfrac,accents}
\usepackage{bbm,wasysym,stmaryrd}
\usepackage{hyperref,color}
\usepackage[normalem]{ulem}
\usepackage{graphicx}
\usepackage[nocompress]{cite}
\usepackage{verbatim}
\numberwithin{equation}{section}

\renewcommand{\P}{\mathbbm{P}}

\newcommand{\Ex}[1]{\mathbb{E}\left[#1\right]} 

\newcommand{\R}{\mathbbm{R}}
\newcommand{\N}{\mathbbm{N}}

\newcommand{\M}{\mathcal{M}}
\newcommand{\C}{\mathcal{C}}

\theoremstyle{plain}
\newtheorem{theorem}{Theorem}[section]
\newtheorem{corollary}[theorem]{Corollary}

\newtheorem{lemma}[theorem]{Lemma}
\newtheorem{definition}[theorem]{Definition}

\newtheorem{remark}[theorem]{\it Remark\/}

\renewcommand{\d}{{\rm d}}

\newcommand\ignore[1]{}

\newcommand{\pnorm}[2]{\left\|#2\right\|_{#1}}
\newcommand{\supnorm}[1]{\pnorm{\infty}{#1}}

\def\sC{\mathcal{C}}
\def\sF{\mathcal{F}}

\def\sK{\mathcal{K}}\def\sL{\mathcal{L}}
\def\sM{\mathcal{M}}\def\sO{\mathcal{O}}
\def\sQ{\mathcal{Q}}

\def\M{\mathcal{M}}

\def\C{\mathcal{C}}

\def\num{N}

\begin{document}
\title{Large deviations for interacting diffusions with path-dependent McKean-Vlasov limit}
\date{\today}
\author{Rangel Baldasso\footnote{Email: \ rangel.bal@gmail.com; \ Bar-Ilan University, 5290002, Ramat Gan, Israel}\and Alan Pereira\footnote{Email: \ alan.anderson.math@gmail.com; \ Federal University of Alagoas, Macei\'{o}, Alagoas, Brazil} \and Guilherme Reis\footnote{Email:  \ ghreis@impa.br; \ Federal University of Bahia, Salvador, Bahia, Brazil}}
\maketitle

\begin{abstract}
We consider a mean-field system of path-dependent stochastic interacting diffusions in random media over a finite time window. The interaction term is given as a function of the empirical measure and is allowed to be non-linear and path dependent. We prove that the sequence of empirical measures of the full trajectories satisfies a large deviation principle with explicit rate function. The minimizer of the rate function is characterized  as the path-dependent McKean-Vlasov diffusion associated to the system. As corollary, we obtain a strong law of large numbers for the sequence of empirical measures. The proof is based on a decoupling technique by associating to the system a convenient family of product measures. To illustrate, we apply our results for the delayed stochastic Kuramoto model  and for a SDE version of Galves-L\"ocherbach model.
\end{abstract}


\section{Introduction}
~
\par Systems of interacting diffusions subject to random media have attracted great attention in statistical physics and have proven to be a fruitful model for neuronal networks.
\par In this paper, we consider interacting diffusions $(\theta^{i,\omega})_{1\leq i \leq N}$ modeled by mean-field systems of It\^o stochastic differential equations (SDEs) subject to random media $(\omega^i)_{1\leq i\leq N}$
\[\d \theta^{i,\omega}(t) = f\big(t, \theta_t^{i,\omega}, L_{N}, \omega^{i}\big)\,\d t+ h\big(\omega^{i}\big)\,\d B^{i}(t), \, t \in [0,T],\]
where $(B^i)_{1\leq i\leq N}$ are i.i.d Brownian motions and $L_N$ is the double-layer empirical measure
\begin{equation}\label{eq:emp_measures_intro}
L_{N}(\theta^{\omega}_T, \vec\omega) = \frac{1}{N} \sum_{i=1}^{N} \delta_{( \theta^{i,\omega}_T, \, \omega^{i})}.
\end{equation}
Here $N$ is the size of the system and $\theta^{i,\omega}_T$ describes the path of particle $i$ in the time interval $[-\tau,T].$

While the interaction between particles takes place only on the drift term, the diffusive terms of the system of SDEs are allowed to depend on an external random field. Besides, we are able to consider path-dependent interactions, where each bit can depend on the entire past of the system. Also, we assume that the interaction is a function of the empirical measure that is bounded and can be non-linear. We define the model precisely in Section~\ref{sec:mean-field-model}.

\par Our goal here is two-fold. First, we perform a large deviations analysis for the double-layer empirical measure $L_N.$  In this case, we prove that the collection $L_{N}$ satisfies a large deviation principle with an explicit rate function.

\par Second, we study the collection of minimizers of the rate function, and prove that the evolution of a typical particle can be described by the solution of a path-dependent McKean-Vlasov equation.

\par Our main results are the following (see Section~\ref{sec:main}).
\begin{enumerate}
\item \label{t:ldp_intro} We prove an annealed large deviation principle for $L_N$ as $N\to \infty$ and obtain an explicit representation for the rate function $H$.
\item \label{t:mckean_vlasov_intro} We verify that the rate function $H$ has a unique minimizer given by a solution to the annealed path-dependent McKean-Vlasov diffusion.  This object is also characterized as the solution of the path-dependent McKean-Vlasov PDE.
\end{enumerate}
The results above are precisely stated as Theorems~\ref{t:ldp} and~\ref{thm:minimizer}, respectively. As an immediate corollary of the two previous statements, we obtain a law of large numbers for the empirical measures $(L_N)_{N \in \N}$, stating that they converge to the unique minimizer of the rate function $H$. In Subsection~\ref{subsec:examples}, we apply our results for the delayed stochastic Kuramoto model and for a SDE version of the Galves-L\"ocherbach model.

\par To the best of our knowledge, this is the first paper that proves large deviations for systems of interacting diffusions with interactions that depend on the past of the particle together with the past of the empirical measure.

\bigskip

\noindent \textbf{Related works.} There is a vast literature concerning large deviations for systems of interacting diffusions. An important example of such systems is when the interaction takes place instantaneously in the sense that the evolution of one diffusion at a certain time depends on other diffusions at the same instant of time. Certain attention was given to remove the mean-field assumption in order to consider more realistic interactions modeled by introducing random strengths or random graphs. We here give a partial review of some works and point the main differences between them and the model we consider here.

\par There are two main types of strategies to prove large deviations for the kind of systems we consider. The first one relies on finding good approximations of the original system. The other exploits the use of Stroock and Varadhan's martingale problem to the Markov process given by the solutions. Here, we use the first strategy.

\par The idea of using some comparison argument to control dependencies is already somewhat established. The nature of the approximation used varies according to the model considered. Budhiraja, Dupuis, and Fischer~\cite{bdf} consider controlled versions of the model and, under very general hypotheses, establish a large deviation principle for the empirical measure of the solution at time $t\in [0,T]$, with rate function that is given as a solution of a variational problem. In their model, they assume more relaxed conditions on the coefficients and allow the diffusion coefficient to also depend on the interaction. In \cite[Section~7.2]{bdf} they consider path-dependent SDEs with the stronger assumption of instantaneous dependence on the empirical measure. Our contribution in that setting is to allow a full path-dependency and obtain a result which is valid for the empirical measure of the solution in the whole time interval $[-\tau,T]$ with an easily interpretable rate function. Also, we consider the diffusions defined in a random environment. Our approach requires more restrictive hypotheses on the coefficients, mainly due to the use of Girsanov's Theorem. For example, we are not able to consider interactions on the diffusive term and, in order to apply Novikov's Condition, we assume that the function modelling the drift interaction is bounded. 

\par Our approximation is closer to the one considered by Dawson and G{\"a}rtner~\cite{dg}, whose model is not path dependent. Their proof is a combination of both techniques mentioned above and do not apply to our case because we lose the Markov property by considering path-dependent diffusions.

\par An approach that resembles ours is used in Dai Pra and den Hollander~\cite{Pra1996}, where the authors derive a large deviation principle for Hamiltonian dynamics through the use of Varadhan's Lemma. In our case, the regularity conditions required by Varadhan's Lemma are not met and we need to directly compare the original system to its approximation. Also, they are in the setting of instantaneous interactions given by a linear function of the empirical measure. Meanwhile, our assumptions are more general, allowing any bounded Lipschitz  dependency on the empirical measure.

\par  Lu\c{c}on~\cite{lucon} combines the techniques from~\cite{Pra1996} and~\cite{dg} to derive a quenched large deviation principle for a Hamiltonian dynamics, when the media variables $\omega$ are fixed. In this case, the idea is again to derive the result via Varadhan's Lemma. The main difficulty is in establishing the large deviation principle for the decoupled model, since one cannot directly apply Sanov's Theorem for the quenched case. The techniques from~\cite{dg} come in hand when proving a quenched version of Sanov's Theorem. We emphasize that we prove only an annealed large deviation principle. 

\par Cabana and Touboul~\cite{cabana_touboul_2018} analyze randomly connected neural networks via delayed interacting diffusions  with independent random interactions. It is a particular case of path dependence, similar to the spin-glass Langevin dynamics considered in works as Ben Arous and Guionnet~\cite{bag} and Guionnet~\cite{guionnet}. We remark that the dependence considered in~\cite{cabana_touboul_2018} is not on the entire past trajectory of the process, but rather on a delayed time which is given by a bounded Lipschitz function (depending on random media variables).

\par Let us shortly hightlight a few works that treat path-dependent McKean-Vlasov equations under different light. Assuming Lipschitz coefficients, Huang~\cite{huang} proves the existence and uniqueness of strong solutions for such equations. Mehri, Scheutzow, Stannat, and Zangeneh~\cite{mssz} consider the more general case that includes jumps and prove the existence of strong solutions and propagation of chaos results under general assumptions of the functions controlling the interactions.

\par Even though we are not aware of works that consider path-dependent systems on other graphs, let us briefly discuss some works that consider different underlying graphs restricted to usual intantaneous dependence. By generalizing the approach of~\cite{dg}, M\"{u}ller~\cite{muller} considers the case when the underlying dependence graph is given by a $d$-dimensional torus. Lu\c{c}on and Stannat~\cite{lucon2016, lucon2014} study convergence and fluctuations for similar models on the integer lattice, with decaying long range interactions.

\par When the underlying graph is chosen at random, Delattre, Giacomin and Lu\c{c}on~\cite{Delattre2016} establish bounds on the distance between the solutions of the system in the mean-field case and over the Erd\"{o}s-R\'{e}nyi random graph, provided the mean degree diverges with logarithmic speed. Without any condition on the velocity of divergence of the mean degree, Oliveira and Reis~\cite{Reis2018} provide large deviation estimates. Lacker, Ramanan and Wu~\cite{Kavita2019} and Oliveira, Reis and Stolerman~\cite{OliveiraReisStolerman2018} study the case of constant average degree and deduce convergence of the solutions to the properly defined model on the Galton-Watson random tree.

\bigskip

\noindent \textbf{Proof overview.} The proof of Statement~\ref{t:ldp_intro} is somewhat similar to the general approach proposed by Dai Pra and den Hollander~\cite{Pra1996}. Even though we are not able to apply Varadhan's Lemma, the rate function one guesses from this approach is still the correct one. In order to circumvent the new difficulties, we use a collection of reference product measures that are known to satisfy a large deviation principle.  These product measures appear naturally in our model as the solution of the uncoupled system.  

\par We draw intuition from the work of Cabana and Touboul \cite{cabana_touboul_2018} to find the  collection of local approximations to the original systems. The goal of these approximations is two-fold: first, they help us to control the intrinsic dependencies of the model. Second, since they are obtained as product measures, as a consequence of Sanov's Theorem, it is easily verified that each of them satisfies a large deviation principle with a relatively simple rate function. 

\par With these approximations in hand, we proceed to deduce the large deviation principle for the original process. There are two main steps one needs to verify (see Section~\ref{sec:fundamental_lemmas}). We first establish a relation between the rate function of the original model and the approximations (Lemma~\ref{lemma:rate_function}). Once this is done, we provide a comparison lemma that relates the Radon-Nykodim derivatives of the two models (Lemma~\ref{lemma:estimates}). Combining both results, we are able to conclude that the empirical measures~\eqref{eq:emp_measures_intro} satisfy a large deviation principle. We remark that both central lemmas rely on estimates  using convenient exponential martingales.

\par Let us now briefly turn our attention to the set of minimizers of the rate function $H,$ Statement~\ref{t:mckean_vlasov_intro}. Once we know Statement~\ref{t:ldp_intro}, we deduce that the rate function $H$ is good, the set of minimizers of $H$ is non-empty, and any minimizer $\mu$ satisfies $H(\mu)=0$.  To prove that any such minimizer is a solution of the annealed McKean-Vlasov equation, we employ Lemma \ref{lemma:rate_function}  that characterizes the rate function of Statement~\ref{t:ldp_intro}. Uniqueness follows from an application of Banach's fixed-point Theorem.

\bigskip

\noindent \textbf{Organization of the paper.} In Section~\ref{sub:preliminaries}, we introduce notations and collect some classical results we will use during the rest of the paper. In Section~\ref{sec:mean-field-model}, we introduce the model we consider in its full generality, and state the complete versions of Statements~\ref{t:ldp_intro} and~\ref{t:mckean_vlasov_intro} in Section~\ref{sec:main}. Section~\ref{sec:fundamental_lemmas} contains the statement of the two main lemmas used in the proof of our results. The proof of these two lemmas are split into Sections~\ref{subsec:rate_function} and~\ref{sec:fundamental_estimates}. In Section~\ref{sec:proof_wldp}, we conclude the proof of the large deviation principle. The proof of Statement~\ref{t:mckean_vlasov_intro} is the content of Section~\ref{sec:minimizer}.

\bigskip

\noindent \textbf{Acknowledgments.} The authors thank Milton Jara and Roberto Oliveira for fruitful discussions during the elaboration of this work. RB is supported by the Israel Science Foundation through grant 575/16 and by the German Israeli Foundation through grant I-1363-304.6/2016. AP was partially supported by Capes/PNPD fellowship 88882.315944/2019-01. GR is supported by a Capes/PNPD fellowship 888887.313738/2019-00. The authors thank IMPA for hospitality and financial support in the early stages of the work.

\section{Preliminaries}\label{sub:preliminaries}
~
\par In this section we fix notation and briefly review some important concepts that will be used in the reminder of the text.

\subsection{Notation}
~
\par Throughout the text, let $\N$ denote the set of nonnegative integers. For $n\in\N\backslash\{0\}$, we write $[n]:=\{1,\dots,n\}$.

\par Let $(S,{\rm d})$ be a Polish space. We write $C(S)$ for the set of bounded continuous functions $\phi:S\to \R$ endowed with the uniform norm $\pnorm{\infty}{\phi} :=\sup_{x\in S}|\phi(x)|$.

\par For a Lipschitz function $\phi:S \to \R$, let
\begin{equation}
\pnorm{\text{Lip}}{\phi} :=\sup_{x,y\in S: x \neq y}\frac{|\phi(x)-\phi(y)|}{{\rm d}(x,y)}
\end{equation}
denote the Lipschitz constant of $\phi$. If $\phi$ is bounded and Lipschitz we define its BL-norm by $\pnorm{\text{BL}}{\phi}=\supnorm{\phi}+\pnorm{\text{Lip}}{\phi}.$

\par For the special case that $S\subset \R$ with $S=[a,b]$, we write $C([a,b])=\sC_{a}^{b}$ and, to avoid confusion when dealing with different intervals, we sometimes write $\pnorm{[a,b]}{\phi}$ for the norm of a function $\phi$ in $\sC_{a}^{b}$.

\par Given an element $\phi \in \sC_{a}^{b}$ and $ t \in [a,b]$, we denote by $\phi_{t}$ the restriction of $\phi$ to $[a,t]$ and by $\phi(t)$ the evaluation of $\phi$ at $t$.

\par If $(S, \d_{S})$ and $(\tilde{S}, \d_{\tilde{S}})$ are Polish spaces, unless otherwise stated, we endow the space $S \times \tilde{S}$ with the metric
\begin{equation*}
\d\big((s_{1}, \tilde{s}_{1}), (s_{2}, \tilde{s}_{2}) \big) = \d_{S}(s_{1}, s_{2}) + \d_{\tilde{S}}(\tilde{s}_{1}, \tilde{s}_{2}).
\end{equation*}
\par Let $\M_{1}(S)$ denote the set of probability measures over (the Borel sets of) $S$. If $X \in S$ is a random element, we denote $\delta_X \in \M_{1}(S)$ the Dirac delta measure at $X$, which can be seen as a random measure in $\M_{1}(S)$. Given a measure $\mu \in \M_{1}(S)$, we write $X \sim \mu$ if $X$ has distribution $\mu$.

The topology of weak convergence in $\sM_1(S)$ is metrized by the bounded-Lipschitz metric defined for $\mu,\nu\in\sM_1(S)$ as
\[{\rm d}_{BL}(\mu,\nu):= \sup\left\{\int \phi \, \d\mu-\int \phi \, \d\nu \,:\, \phi:S\to  \R,\, \pnorm{BL}{\phi}\leq 1\right\}.\]
A proof of this fact can be found in~\cite[Section 8.3]{bogachev}.

If $X$ and $Y$ are random elements in $S$ defined on the same probability space and $X\sim \mu$ and $Y\sim \nu$ then
\begin{eqnarray} \label{eq:boundonBL} 
{\rm d}_{BL}(\mu,\nu)\leq \Ex{\min\{{\rm d}(X,Y),\,2\}}.
\end{eqnarray}

\par For two probability measures $\mu, \nu \in \M_{1}(S)$, the relative entropy between $\mu$ and $\nu$ is the quantity
\begin{equation}\label{eq:relative_entropy}
I(\mu|\nu) = \sup \left\{ \int \phi \,\d \mu - \log \int \exp \{ \phi \} \,\d \nu \right\},
\end{equation}
where the supremum above is taken over all bounded functions $\phi:S \to \R$.

\par Finally, let us collect some facts about the relative entropy. We refer the reader to the Appendix of~\cite{kipnis_landim} for a more detailed description and proofs.

\par By considering constant functions, we have $I(\mu|\nu) \geq 0$. Besides, if $I(\mu|\nu)=0$, then $\mu=\nu$. However, the relative entropy between two probability measures is not always finite. In fact, $I(\mu|\nu)$ is finite if, and only if, $\mu$ is absolutely continuous with respect to $\nu$.

\begin{lemma}[Entropy Inequality] \label{lemma:entropy-inequality}
For any measurable function $g:S \to \R$ that is either $\mu$-integrable or bounded from above or below, we have
\begin{equation}\label{eq:entropy_inequality}
\int g \,\d \mu \leq I(\mu|\nu)+ \log \int \exp \{ g \} \,\d \nu,
\end{equation}
\end{lemma}

\begin{proof}
Directly from the definition, one obtains~\eqref{eq:entropy_inequality} for any bounded function $g:S \to \R$. The monotone convergence theorem can be applied to conclude that the inequality above also holds for measurable functions that are bounded only from above or below. It remains to verify the claim for the case when $g$ is $\mu$-integrable. In this case, let $g_{M}=\min\{g, M\}$ and observe that
\begin{equation*}
\int g_{M} \,\d \mu \leq I(\mu|\nu)+ \log \int \exp \{ g_{M} \} \,\d \nu.
\end{equation*}
We now use the dominated convergence theorem to justify that $\int g_{M} \, \d \mu \to \int g \, \d \mu$ and the monotone convergence theorem to conclude that $\int \exp \{ g_{M} \} \,\d \nu \to \int \exp \{ g \} \,\d \nu$, as $M$ grows. This concludes the proof.
\end{proof}

\subsection{Metrics on $\sC_{-\tau}^{t}\times \R^{d}$}\label{subsec:metrics}
~
\par We will usually work with objects that take values on the space of measures $\M_{1}\left(\sC_{-\tau}^{t} \times \R^{d} \right)$, where $\tau>0$ is a fixed constant and $t\in [0,T]$. For this reason, it will be useful to define a proper metric in this space.

\par On $\sC_{-\tau}^{t} \times \R^{d}$, we consider the metric given by
\begin{equation}\label{eq:metric_product}
\text{dist}^{t}\big((x_{t},\omega),(y_{t},\bar{\omega})\big) = \pnorm{[-\tau, t]}{x_{t}-y_{t}}+\pnorm{2}{\omega-\bar{\omega}},
\end{equation}
where $\pnorm{2}{\cdot}$ denotes the usual Euclidean metric on $\R^{d}$.

\par Given this metric, for two measures $\mu$ and $\nu$ on $\M_{1}(\sC_{-\tau}^{t} \times \R^{d})$ we write
$\d^t_{\text{BL}}(\mu,\nu)$
for the BL-distance depending on $t$.

\par Throughout the text, we use an abuse of notation whenever considering the push-forward of the measures by the projection map $\pi_{t}:\sC_{-\tau}^{T} \times \R^{d} \to \sC_{-\tau}^{t} \times \R^{d}$. For $\mu, \nu \in \sM_{1}(\sC_{-\tau}^{T} \times \R^{d})$, we will write
\begin{equation}\label{eq:metric_measures}
\d^{t}_{\text{BL}}\left(\mu, \nu\right) = \d^{t}_{\text{BL}}\left(\mu\circ \pi_{t}^{-1}, \nu \circ \pi_{t}^{-1}\right).
\end{equation}

\begin{remark}\label{remark:metric}
\rm Our techniques still hold if one chooses to replace the metrics $\text{dist}^{t}$ in~\eqref{eq:metric_product} with any other collection of metrics $\d^{t}$ that still make $\sC_{-\tau}^{t} \times \R^{d}$ a Polish space and such that $t \mapsto \d^{t}\big((x_{t}, \omega), (y_{t}, \tilde{\omega})\big)$ is non-decreasing, for any two functions $x, y \in \sC_{-\tau}^{T}$ and two vectors $\omega, \tilde{\omega} \in \R^{d}$. This might be useful when considering different types of interactions in our general model, as we shall see in Subsection~\ref{subsec:examples}.
\end{remark}

\subsection{Large deviation principle}\label{subsec:ldp}
~
\par We recall the definition of large deviation principle (LDP), in its weak and strong forms, and state a classical result that shows that a weak LDP together with exponential tightness implies a strong LDP  (cf.~\cite[Section~1.2]{dembo2009large}). 

Since all LDP considered in this work will hold with speed $N$ we will omit this information.

\begin{definition}
A function  $I:S \to [0,\infty]$ is called a rate function if  $I \not \equiv \infty$ and $I$ is a lower semicontinuous map. A rate function is said good if the level sets $I^{-1}([0,a])$ are compact.
\end{definition}

\begin{definition}[Large deviation principle]\label{def:LDP} A sequence $\{\P_N\}_{N\in\N}$ of probabilities over $S$  satisfies a large deviation principle (LDP) with rate function $I$ if

\begin{enumerate}
\item (Lower bound) For any open set $\sO \subset S$,
\begin{equation*}
\displaystyle\liminf_{N \to \infty}\dfrac{1}{N} \log  \P_N(\sO) \geq -\inf_{x \in \sO}I(x);
\end{equation*}

\item (Upper bound) For any closed set $\sC \subset S$,
\begin{equation*}
\displaystyle\limsup_{N \to \infty}\dfrac{1}{N}\log  \P_N(\sC) \leq -\inf_{x \in \sC}I(x).
\end{equation*}
\end{enumerate}
\end{definition}

\begin{definition}[Weak large deviation principle]\label{def:WLDP}
A sequence  $\{\P_N\}_{N\in\N}$ of probabilities over $S$ satisfies a weak large deviation principle (wLDP) with rate function $I$ if it satisfies the lower bound in Definition~\ref{def:LDP} for open sets and the upper bound for all compact sets $\sK \subset S$.
\end{definition}

\begin{definition}
A sequence of probability measures $\{\P_N\}_{N\in\N}$ on $S$ is exponentially tight if, for every $\alpha < \infty$, there exists a compact set $K_{\alpha} \subset S$ such that
\begin{equation*}
\limsup_{N\to\infty}\frac{1}{N} \log \mathbb{P}_N\left(K_{\alpha}^{\complement}\right) < -\alpha.
\end{equation*}
\end{definition}

\par The importance of exponential tightness lies in the following result, that says that, in order to prove a LDP for an exponentially tight family of probabilities, it suffices to conclude a wLDP.

\begin{theorem}\label{thm:wLDP+expt=LDP}
If an exponentially tight sequence of probability measures satisfies a wLDP with a rate function $I$, then $I$ is a good rate function and the LDP holds.
\end{theorem}

\section{The mean-field model and related objects}\label{sec:mean-field-model}
~
\par In this section we fully specify the interacting diffusion model we will consider. We begin with some definitions that we will use, and introduce the complete model. In the next subsection, we list all the technical assumptions we make in order to prove our theorem and finish the section by introducing the decoupling family and precisely defining the path-dependent McKean-Vlasov diffusions.

\subsection{Definition}\label{subsec:definition}
~
\par In order to precisely define the model, we first introduce some notation. Throughout the text, $T >0$ is a fixed time horizon and $\tau>0$ is a finite constant that bounds how much dependence of the past one can have at time zero. Consider also the following objects.
\begin{enumerate}
\item A probability distribution $\mu_{0}$ over $\C_{-\tau}^{0}$, for the {\em initial states of the diffusions}.
\item A probability distribution $\mu_{\rm med}$ over $\R^d$, for the {\em media variables}.
\item A function $f:[0,T] \times \sC_{-\tau}^{T} \times \sM_{1}\left( \sC_{-\tau}^{T} \times \R^{d}\right) \times \R^{d} \to \R$ that determines {\em interactions between particles}. These terms will depend on the state of the diffusions and on their media variables.
\item A function $h:\R^d\to [0,+\infty)$ that determines the {\em single-particle diffusion term} in our interacting model. 
\end{enumerate}   

\par We postpone the introduction of the technical conditions we impose to these objects to the next subsection.

Start with a probability space together with a filtration $(\Omega,(\sF_t)_{t\in [-\tau,T]},\P)$. Let $\mathcal{W}$ denote the standard Wiener measure over $\sC_0^T$. To define our model for each $\num\in\N$, we assume that we have the following random variables defined in $(\Omega,(\sF_t)_{t\in [-\tau,T]},\P)$:
\begin{enumerate}
\item $\big(\xi_{0}^{i}\big)_{i\in [\num]} \sim {\mu_{0}}^{\otimes \num}$,
\item $\big(\omega^{i}\big)_{i\in [N]} \sim \mu_{\rm med}^{\otimes \num}$,
\item  $\big(B^{i}\big)_{i \in [\num]} \sim \mathcal{W}^{\otimes \num}$.
\end{enumerate}
As to measurability assumptions, we assume that the random variables $\xi_{0}^{i}$ are measurable with respect to $\sF_{0}$, that each $\omega^{i}$ is measurable with respect to $\sF_{-\tau}$, and that the collection of Brownian motions is adapted to the filtration $(\sF_t)_{t\in [0,T]}$. Furthermore, we also assume that the collections of random variables above are mutually independent. Notice also that, for each $i \in [N]$, $\xi_{0}^{i}:[-\tau,0] \to \R$ is a bounded continuous function.

\par The system of interacting diffusions we consider is given by the following definition. 

\begin{definition}\label{def:process} In the previous setting, let $N$ be fixed and consider a realization $\omega=\big(\omega^{i}\big)_{i\in [N]}$. We define the system of interacting diffusions $\theta^{\omega}=\left(\theta^{i,\omega}\right)_{i\in [N]}$ as the strong solution of the system of It\^o Stochastic Differential Equations (SDEs) given by, for $i \in [N]$,
\begin{equation}\label{eq:intdiff}
\begin{cases} \d \theta^{i,\omega}(t) = f\big(t, \theta^{i,\omega}_t, L_{N}, \omega^{i}\big)\,\d t+ h\big(\omega^{i}\big)\,\d B^{i}(t), & 0 \leq t \leq T,\\
\theta^{i,\omega}_{0}=\xi^{i}_{0},
\end{cases}
\end{equation}
where $L_{N}=L_{N}(\theta^{\omega}, \omega)$ is the (random) empirical measure on the space $\sC_{-\tau}^{T} \times \R^{d}$ given by
\begin{equation*}
L_{N}(\theta^{\omega}, \omega) = \frac{1}{N} \sum_{i=1}^{N} \delta_{\left( \theta^{i,\omega}, \, \omega^{i}\right)}.
\end{equation*}
\end{definition}

\par The double-layer empirical measure $L_{N}(\theta^{\omega}, \omega)$ introduced above is a measure on the space $\sC_{-\tau}^{T} \times \R^{d}$ and may be seen as a function of the solution of the system of SDEs~\eqref{eq:intdiff}.

\par We write $Q_{N}^{\omega}\in \M_1\big((\sC_{-\tau}^{T})^N\big)$ for the law of the solution of the system~\eqref{eq:intdiff} (taking into account the randomness of the initial condition) for a fixed collection of values $\omega=\big(\omega^{i}\big)_{i\in [N]}$. We define $Q_N\in \M_1\big((\sC_{0}^{T}\times \R^d)^N\big)$ through its expression on rectangles as
\begin{equation}
Q_N(A\times B)=\int_{B}Q_{N}^{\omega}(A) \, \d \mu_{\rm med}^{\otimes N}(\omega).
\end{equation}
We call $Q_N$ the annealed law and $Q_{N}^\omega$ the quenched law of $\theta^{\omega},$ respectively.

\par It is not always the case that the system~\eqref{eq:intdiff} has a strong solution. The hypotheses we assume on the functions $f$ and $h$ will imply this. We will also present examples of cases that fall under our hypotheses.

\begin{remark}
\rm The object we are interested in is $Q_N(L_N\in \cdot\,)$, which, for each $N \geq 1$, is an element of the space $\M_1\Big(\M_1\big(\sC_{-\tau}^{T}\times \R^d\big)\Big)$.
\end{remark}

\subsection{General assumptions}\label{subsec:assumptions}
~
\par In this subsection we list the collection of hypotheses we assume. In Subsection~\ref{subsec:examples}, we give examples of cases that fall under our assumptions.

\par A random measure $\nu \in \sM_1(\sC_{-\tau}^T\times \R^d)$ can be seen as a process $(\nu(t))_{t \in [-\tau,T]}$, where $\nu(t)$ is the push-forward through the canonical projection $\pi_{t}:\sC_{-\tau}^{T} \times \R^{d} \to \sC_{-\tau}^{t} \times \R^{d}$. We say that $\nu$ is adapted if it is measurable with respect to $\sF_t$, for all $t$.
 
\par About the function $f$, we assume that it is adapted in the sense that, for each pair of adapted random elements $(X,\nu)\in \sC_{-\tau}^T\times \sM_1(\sC_{-\tau}^T\times \R^d)$, the random variable $f(t,X,\nu,\omega)$ is measurable with respect to $\sF_t$ for any choice of $\omega \in \R^d.$   In the examples, it will be the case that the function $f$ will depend only on $X_t=X_{[-\tau,t]}$ and on $\nu_t=(\nu(s))_{s\in [-\tau,t]}.$ We will write $f(t,X_t,\nu,\omega)$ to emphasize this dependence.
\par Furthermore, we assume that, for fixed $t\in [0,T]$ and $\omega\in \R^d$, the function $f(t,\,\cdot\,,\,\cdot\,,\omega)$ is Lipschitz and the Lipschitz constant is uniform on $(t,\omega).$ More specifically, there exists a positive constant $\pnorm{SL}{f} \geq 0$ such that, for all $t \in [0,T]$ and $\omega \in \R^{d}$,
\begin{equation}	\label{equation:f-superlipschitz}
\left|f(t,x,\mu,\omega)-f(t,y,\nu,\omega)\right| \leq \pnorm{SL}{f} \left( \pnorm{[-\tau,t]}{x_{t}-y_{t}} + \d^{t}_{\text{BL}}\left(\mu, \nu\right)\right),
\end{equation}
where $\d^{t}_{\text{BL}}$ is the BL-distance introduced in~\eqref{eq:metric_measures}.  This constant $\pnorm{SL}{f}$ resembles the definition of the Lipschitz constant. However, this is not exactly the case, since, for each $t \in [0,T]$, we consider different metrics on the RHS of the equation above. If one thinks of $f$ as a family of functions indexed by $t \in [0, T]$, the condition above reduces to requiring that all functions are Lipschitz with uniformly bounded constant. We also assume that $f$ is a bounded function.

\par Regarding the function $h$, we assume it is a bounded Lipschitz function. Besides, we suppose that $h$ is uniformly lower bounded by a positive constant $h_{*}>0$.

Under these hypotheses, \cite[Chapter~IX, Theorem~2.4]{revuz_yor} implies that, for any $\omega=\big(\omega^{i}\big)_{i\in [N]}$ and initial conditions $(\xi^{i}_{0})_{i \in [N]}$, the system~\eqref{eq:intdiff} admits a unique strong solution $\theta^{\omega}$.

\subsection{Examples}\label{subsec:examples}

\subsubsection{Delayed interacting diffusions}
~
\par Our results relate naturally with a generalization of Kuramoto model, where delays are introduced and whose system is given by
\begin{equation}\label{eq:intdiff_ex}
\d \theta^{i,\omega}(t) = \frac{1}{N}\displaystyle\sum_{j=1}^N F\big(\theta^{i,\omega}(t),\;\theta^{j,\omega}(t-\bar{\tau}(\omega^i,\omega^j)),\;\omega^i,\;\omega^j\big)\,\d t + h\big(\omega^i\big)\,\d B^i(t),
\end{equation}
where $F:\R\times \R\times \R^d\times \R^d\to \R$, $h:\R^d\times\R^d\to [0,+\infty)$, and $\bar{\tau}:\R^d \times \R^d \to [0,\tau] $ are bounded Lipschitz functions. One can interpret the field $(\omega^{i})_{i \in [N]}$ as the positions of the particles, and the delay is a function of said positions.

\par We recover our original model by setting
\begin{equation}
f\big(t, x, \mu, \omega \big) = \int F \big(x(t), y(t-\tau(\omega, \sigma)), \omega, \sigma \big) \, \d \mu \big (y, \sigma).
\end{equation}

\par The delayed Kuramoto model is given by the choice
\begin{equation*}
F(x,y,\omega, \sigma) = \sin(x-y)+\omega,
\end{equation*}
and our result applies when one assumes that $\omega$ has a compactly supported distribution.

\par Let us now verify that the function $f$ defined above satisfies the conditions required in Subsection~\ref{subsec:assumptions}. We employ Remark~\ref{remark:metric}, and use the metric
\begin{equation}\label{eq:metric_1}
\begin{split}
d_{K,t}((x,\omega), &(x', \omega'))^{2} = |\omega-\omega'|^{2} \\
& +\sup\left\{|x(s)-x'(u)|: s, u \in [0,t] \text{ and } |s-u| \leq K|\omega-\omega'| \right\}^{2}.
\end{split}
\end{equation}
Cabana and Touboul~\cite{cabana_touboul_2018} prove in Remark 6 that each $d_{K,t}$ is, in fact, a distance and that $(\C_{-\tau}^t \times \R^d, d_{K,t})$ is a Polish space.

\par Fix two measures $\mu, \nu \in \M_1\big(\sC_{-\tau}^{T}\times \R^d\big)$, two functions $x, y \in \sC_{-\tau}^{T}$, a time $t \in [0,T],$ and a vector $\omega \in \R^d$. In order to verify Inequality~\eqref{equation:f-superlipschitz}, choose $K = \pnorm{\text{Lip}}{\tau}$. It is easy to see that the function $\phi(z_T,\sigma)=F\big(x(t),z(t-\tau(\omega,\sigma)), \omega, \sigma \big)$ is Lipschitz if we use the metric \eqref{eq:metric_1}. Then
\begin{align*}
&|f\big(t, x, \mu, \omega \big) -f\big(t, y, \nu, \omega \big)| \\ 
\leq& \bigg|\int F\big(x(t),z(t-\tau(\omega,\sigma)), \omega, \sigma \big)\d \mu(z_T,\sigma)\\
&
-\int F\big(y(t),z(t-\tau(\omega,\sigma)), \omega, \sigma \big)\d \nu(z_T,\sigma)\bigg|\\
=&\bigg|\int F\big(x(t),z(t-\tau(\omega,\sigma)), \omega, \sigma \big)\d (\mu-\nu)(z_T,\sigma)\\
&+\int \big[F\big(x(t),z(t-\tau(\omega,\sigma),\omega,\sigma \big)-F\big(y(t),z(t-\tau(\omega,\sigma)\big), \omega, \sigma \big)\big]\d \nu(z_T,\sigma)\bigg|.\\
\leq&   \pnorm{\text{BL}}{F} \Big(\d^{t}_{\text{BL}}\left(\mu, \nu \right)+ \pnorm{[-\tau,t]}{x-y}   \Big),
\end{align*}
as required.

\subsubsection{SDE mean-field version of Galves-L\"ocherbach model}
~
\par Inspired by the systems of interacting chains with memory of variable length (cf.~\cite{GL}), we consider
\[f(t,\theta^i,L_N,\omega^i)=f(t,L_N,\omega^i)=\frac{1}{N}\sum_{j=1}^N\int_{t-\tau(\omega^i)}^tg(t-s,\omega^j)\phi(\theta^j(s)) \, \d s,\]
where we make the following assumptions.
\begin{enumerate}
\item The media variables $(\omega^i)_{i\in [N]}$ can be parameters associated to $\theta^i$ such as position and other chemical properties.
\item The positive variable $\tau(\omega^i)$ depends on the random media $\omega^i$, and $t-\tau(\omega^i)$ models the time of the last spike of $\theta^i$ before time $t.$ The integral from $t-\tau(\omega^i)$ to $t$ says that we are considering only contributions of $\theta^j(s)$ to the evolution of $\theta^i(t)$ until the last spike of $\theta^i$ before time $t$.
\item $\phi:\R\to [0,1]$ is monotone increasing Lipschitz function.
\item The function $g:\R_+\times \R^d\to\R_+$ is bounded and continuous. This function can model the loss of the contribution of $\theta^j$ due to the time delay and depends on the media variable $\omega^j$.
\end{enumerate} 

We can define, for $\mu \in \sM_1(\sC_{-\tau}^T\times \R^d)$,
\[f(t,\mu,\omega)=\int_{t-\tau(\omega)}^t\int g(t-s,\sigma)\phi(y(s)) \, \d \mu(y_T,\sigma) \, \d s,\]
and it is easy to see that this function satisfies all assumptions of Subsection~\ref{subsec:assumptions}.

\begin{remark}\rm A more realistic variable to model spikes would be $\sL_t(\theta^i)=\sup\{s<t\,:\,\theta^i(s)\geq 1\}.$ In plain words, $\sL_t(\theta^i)$ is the last time the particle $\theta^i$ was above the threshold $1$ (similar to a spike). However our assumptions do not fit the use of $\sL_t(\theta^i)$ since this function is not continuous in time. We remark that this can be a good subject for future study.  
\end{remark}

\subsection{The decoupling family and McKean-Vlasov diffusions}
\label{subsec:decoupling}
~
\par In this subsection, we introduce the main tool we use to control dependencies of the solutions $\theta^{\omega}$ (see Definition~\ref{def:process}).

\par If the measures $Q_{N}$ could be written as product measures, it would be possible to apply Sanov's Theorem to conclude that $(Q_N(L_N \in \cdot\,))_{N\geq 1}$ satisfies a large deviation principle. Unfortunately, this is not the case and to surpass the dependencies of the model we will use an auxiliary family of probability measures, that we call the decoupling family,
\begin{equation}
(Q_\nu)_{\nu\in\sM_1(\sC_{-\tau}^T\times \R^d)}.
\end{equation}

\par Our main strategy is to locally compare $Q_{N}$ to $Q_{\nu}^{\otimes N}$, for suitable choices of $\nu \in \sM(\sC_{-\tau}^{T} \times \R^{d})$. The idea to construct $Q_{\nu}$ is to consider the SDE~\eqref{eq:intdiff} with $L_{N}$ replaced by $\nu$. For $\omega \in \R^{d}$, let $\psi^{\omega, \nu}$ be the solution of
\begin{multline}\label{eq:intdiff-nu}
\begin{cases} \d \psi^{\omega, \nu}(t) =f\big(t, \psi^{\omega, \nu}, \nu, \omega\big)\,\d t+ h\big(\omega\big)\,\d B(t), & 0 \leq t \leq T, \\
\psi^{\omega, \nu}_{0}=\xi_{0}.
\end{cases}
\end{multline}

\begin{definition}\label{def:decoupling}
For any fixed $\omega \in \R^d$, let $Q_\nu^\omega$ denote the law of $\psi^{\omega,\nu}$ and define the measure $Q_\nu \in \sM_1(\sC_{-\tau}^T\times\R^d)$ via its representation on rectangles $A \times B$ as
\begin{equation}\label{eq:decoupling}
Q_\nu(A\times B)=\int_{B}Q_\nu^\omega(A) \, \d \mu_{\rm med}(\omega).
\end{equation}
\end{definition}

\par To conclude this subsection, we define the annealed path-dependent McKean-Vlasov law using the decoupling family of Definition~\ref{def:decoupling}. Let us first provide some heuristics for the definition.

\par Assume that $L_{N}$ converges to the law of a random element $V=(V^{\omega}, \omega)$, whose distribution we denote by $\sL(V)$. At the same time, due to symmetries of the system, one can infer that all paths $(\theta^{i,\omega}, \omega)$ should be equally distributed, and that the limit law should be equal to the distributional limit of $L_{N}$, $\sL(V)$. From this, we can use Equation~\eqref{eq:intdiff} to write
\begin{equation}\label{eq:heuristics-mckeanvlasov}
\d V^{\omega}(t) = f\big(t, V^{\omega}, \sL(V), \omega\big) \,\d t + h(\omega)\,\d B(t).
\end{equation}
The heuristics above suggests that, if $L_N$ converges to $\sL(V)$, then $V^\omega$ satisfies~\eqref{eq:heuristics-mckeanvlasov}. In light of Definition~\ref{def:decoupling}, this means that $Q_{\sL(V)}=\sL(V).$ 
\begin{definition}\label{def:annealedMV}
The annealed path-dependent McKean-Vlasov law associated to the system of SDEs~\eqref{eq:intdiff} is the probability measure $\nu^*\in\sM_1(\sC_{-\tau}^{T}\times\R^d)$ such that
\begin{equation*}
Q_{\nu^*}=\nu^*.
\end{equation*}
\end{definition}

\par In particular, in the definition above, we have $\nu^{*} \circ \pi^{-1}_{0} = \mu_{0}$, where $\pi_{0}:\C_{-\tau}^{T}  \times \R^{d} \to \C_{-\tau}^{0}  \times \R^{d}$ is the canonical projection.

\par The existence of such a measure is not immediately clear. In Theorem~\ref{thm:minimizer} below, we prove its existence and uniqueness, and characterize it as weak solution of a  path-dependent McKean-Vlasov PDE and the unique minimizer of the rate function of the LDP satisfied by the family of probabilities $(Q_N(L_N \in \cdot\,))_{N\geq 1}$. 
\par

\subsection{Heuristics on the LDP and more definitions}
\label{sec:heuristics}
~
\par In this subsection we present a brief heuristics of how one pursuits the proof of a LDP for the sequence $Q_N(L_N\in \cdot\,)$. In parallel, we also motivate and define the rate function and related objects.

\par In order to prove a LDP for the sequence $Q_N(L_N\in \cdot\,)$, we will prove that this sequence is exponentially tight and that it satisfies a weak LDP (cf. Section~\ref{subsec:ldp}).

\par We will first prove the lower bound. The proof of the upper bound in the weak LDP follows a similar argument. In order to prove the lower bound, we need, for each fixed measure $\nu \in \sM_1(\sC_{-\tau}^{T} \times \R^{d})$ and $\delta>0$, to get an exponential lower bound for
\begin{equation}\label{eq:lbgoal}
Q_N(L_N\in B(\nu,\delta)).
\end{equation}
Recall, however, that Sanov's Theorem gives us a lower bound when we replace $Q_N$ by $Q_{\nu}^{\otimes N}$ in~\eqref{eq:lbgoal} above. Since we are considering the event where $L_{N}$ is close to $\nu$, one might expect that it is possible to compare the two probabilities and obtain the necessary bounds for~\eqref{eq:lbgoal}.

\par There are two main steps in formalizing the previous idea.

\par \textbf{First step:} Compare the probabilities $Q_N$ and $Q_\nu^{\otimes N}$, by considering their Radon-Nikodym derivatives with respect to a common reference measure.
\par \textbf{Second step:} Obtain the rate function for $Q_N(L_N \in \cdot\,)$ from the rate functions provided by Sanov's Theorem.

The common reference measure we will use is introduced in the following definition.
\begin{definition}\label{def:P}
For each $\omega \in  \R^{d}$, let $P^{\omega}$ denote the law of the unique strong solution of the SDE
\begin{equation}
\begin{cases} 
\d x^{\omega}(t)= h\big( \omega \big)\,\d B(t), & 0\leq t\leq T \\
x^{\omega}_{0}= \xi_{0}.
\end{cases}
\end{equation}
In other words, $P^{\omega}$ is characterized by the fact that, under $P^{\omega}$, $(x_{t}^{\omega})_{t \in [-\tau,0]}$ is distributed according to $\mu_{0}$ and, for $t \in [0,T]$, $B(t)=\frac{1}{h(\omega)}(x^{\omega}(t)-x^{\omega}(0))$ is a Brownian motion independent of $(x_{t}^{\omega})_{t \in [-\tau,0]}$. Let $P$ be the probability measure on $\sC_{-\tau}^{T} \times \R^d$ given on cylinders by
\begin{equation*}
P(A \times B) = \int_{B}P^{\omega}(A) \,\d\mu_{\rm med}(\omega).
\end{equation*}
\end{definition}

\par We will relate $Q_{\nu}^{\otimes N}$ to $P^{\otimes N}$ via Girsanov's Theorem. In order to do so, let
\begin{equation}\label{def:D}
D^{\nu,\omega}(y_{T}):=\frac{1}{h(\omega)^{2}}\int_{0}^{T}f(t,y_{t}, \nu, \omega)\,\d y(t)- \frac{1}{2h(\omega)^{2}}\int_{0}^{T}f(t, y_{t}, \nu, \omega)^{2} \,\d t.
\end{equation}
In Appendix~\ref{app:radon-nikodym}, we prove by applying Girsanov's Theorem that, almost surely with respect to $P$,
\begin{align}\label{eq:dQnu/dP}
\frac{\d Q_\nu^{\otimes N}}{\d P^{\otimes N}}\big(\vec{x}_{T}, \vec{\omega}\big)=\exp\Big(N\int D^{\nu,\omega}\big(y_{T}\big) \, L_N(\d y_{T}, \d \omega)\Big),
\end{align}
where $L_{N}= L_{N}\big(\vec{x}_{T}, \vec{\omega}\big)$ is the empirical measure of the vector $\big(\vec{x}_{T}, \vec{\omega}\big)$. We also have
\begin{equation}\label{eq:dQN/dP}
\frac{\d Q_N}{\d P^{\otimes N}}\big(\vec{x}_{T}, \vec{\omega}\big)=\exp\Big(N\int D^{L_N,\omega}\big(y_{T}\big) \, L_N( \d y_{T}, \d \omega) \Big).
\end{equation}

\par For the second step in the proof, one might draw inspiration from Varadhan's Lemma. Assume for a moment that the function
\begin{equation}\label{eq:intuition_1}
\mu \in \M_1(\sC_{-\tau}^T\times\R^d)\mapsto \int D^{\nu,\omega}\big(x_{T}\big) \, \d \mu(x_{T},\omega)
\end{equation}
is bounded and continuous. Under these hypotheses, a straightforward application of Varadhan's Lemma combined with~\eqref{eq:dQnu/dP} implies that the sequence of measures $\big(Q_\nu^{\otimes N}(L_N\in \cdot\,)\big)_{N\geq 1}$ satisfies a LDP with rate function
\begin{equation}\label{eq:varadhan_wrong}
\mu\mapsto I(\mu|P)-\int D^{\nu,\omega}\big(x_{T}\big) \, \d \mu(x_{T},\omega).
\end{equation}
Even though~\eqref{eq:intuition_1} is not bounded nor continuous, this intuition leads to the correct answer, as we shall see. In Lemma~\ref{lemma:rate_function}, we argue that the function
\begin{equation}\label{eq:gamma}
\Gamma_\nu(\mu)=\begin{cases}
\displaystyle\int D^{\nu,\omega}\big(x_{T}\big) \, \d \mu(x_{T},\omega), & \text{if } I(\mu|P) < \infty,\\
\infty, & \text{otherwise}
\end{cases}
\end{equation}
is well defined. With this in hands, we are able to prove in Lemma~\ref{lemma:rate_function} that the sequence $\big(Q_\nu^{\otimes N}(L_N\in \cdot\,)\big)_{N\geq 1}$ satisfies a LDP with rate function
\begin{equation}\label{eq:def_of_Hnu}
H_{\nu}(\mu)=\begin{cases}
I(\mu\vert P) - \Gamma_\nu(\mu), & \text{if } I(\mu\vert P) < \infty,\\
\infty, & \text{otherwise},
\end{cases}
\end{equation}
in alignment with Varadhan's Lemma.

\par A similar heuristics applies to the sequence $\big(Q_{N}(L_N\in \cdot\,)\big)_{N\geq 1}$ and we will prove that it satisfies a LDP with rate function $H(\mu)=H_{\mu}(\mu)$, for each $\mu \in \sM_{1}(\sC_{-\tau}^{T}\times \R^{d})$ (see Theorem~\ref{t:ldp}).

\section{Statement of main results}
\label{sec:main}
~
\par We are now in position to state our main results. Our main theorem states that the sequence of empirical measures of the system of SDEs~\eqref{eq:intdiff} satisfies a large deviation principle with a good rate function. We will prove that the rate function has a unique minimizer given by the annealed path-dependent McKean-Vlasov law of Definition~\ref{def:annealedMV}. As a simple consequence of Borel-Cantelli's Lemma, we also have a strong law of large numbers for the empirical measures.

\par Our first result is the following.
\begin{theorem}[Proof in Section~\ref{sec:proof_wldp}]\label{t:ldp} Under the assumptions of Subsection~\ref{subsec:assumptions}, the sequence $(Q_N(L_N \in \cdot\,))_{N \in \N}$ (cf. Subsection~\ref{subsec:definition}) satisfies a LDP with a good rate function given by
\begin{equation}\label{eq:rate_function}
H(\mu)=H_{\mu}(\mu).
\end{equation}
\end{theorem}

\par As mentioned in Subsection~\ref{subsec:decoupling}, the annealed path-dependent McKean-Vlasov law is the natural candidate for the limit of the measures $L_N$. The next theorem fully characterizes the minimizers of $H$ and, in particular, we obtain a strong law of large numbers  for $L_{N}$ (cf. Corollary~\ref{cor:SLLN}).

\begin{theorem}[Proof in Section~\ref{sec:minimizer}]\label{thm:minimizer}
In the settings of Theorem~\ref{t:ldp}, the rate function $H$ has a unique minimizer given by the annealed path-dependent McKean-Vlasov diffusion $\nu^*$ of Definition~\ref{def:annealedMV}. Furthermore, the marginal of $\nu^{*,\omega}$ at time $t\in [0,T]$, $\nu^{*,\omega}(t),$ satisfies, for any bounded function $\phi:\R\to\R$ with continuous bounded derivatives up to order two, (see Remark \ref{r1})
\begin{equation}\label{eq:mvpde}
\int_{\R}\phi(u)(\nu^{*,\omega}(t)-\nu^{*,\omega}(0))(\d u)=
\int_0^t\int_{\sC_{-\tau}^T}L_{\nu^{*,\omega}}(\phi)(s,x_T)\nu^{*,\omega}(\d x_T) \, \d s,
\end{equation}
where the differential operator $L_{\nu^{*,\omega}}$ is defined  by
\[L_{\nu^{*,\omega}}(\phi)(t,x_T)=f(t,x_t,\nu^*,\omega)\phi'(x(t))+\frac{h(\omega)^2}{2}\phi''(x(t)).\]
\end{theorem}
\begin{remark}\label{r1}\rm Equation~\eqref{eq:mvpde} can be seen as the weak formulation of the coupled path-dependent McKean-Vlasov PDEs \begin{equation}\label{mvpde}
\begin{cases}
\partial_t\nu^{*,\omega}(t)=L_{\nu^{*,\omega}}^*\nu^{*,\omega}\\
\nu^{*,\omega}_0=\mu_0.
\end{cases} 
\end{equation}
\end{remark}

\begin{corollary}[Strong law of large numbers, proof omitted] \label{cor:SLLN}
If one couples the measures $Q_N(L_N\in \cdot\,)$ in the same probability space $(\P, \Omega,\sF)$. Then, for any bounded continuous function $F: \sC_{-\tau}^T\times\R^d \to \R$,  the following holds $\P-$almost surely
\begin{equation}
\lim_{N\to\infty} \int F(x_{T}, \omega) \, \d L_{N} = \int F(x_{T},\omega)\, \d \nu^*.
\end{equation}
\end{corollary}

\subsection{Overview of the proofs}
~
\par Let us now briefly describe the proofs of Theorems~\ref{t:ldp} and~\ref{thm:minimizer}.

\par The more demanding result is the full LDP of Theorem~\ref{t:ldp}. According to Theorem~\ref{thm:wLDP+expt=LDP}, in order to conclude Theorem~\ref{t:ldp} it suffices to verify a weak LDP and exponential tightness of the sequence
\begin{equation}\label{eq:seqQN}
\big(Q_N(L_N \in \cdot\,)\big)_{N\in\N}.
\end{equation}

\par We verify exponential tightness in Subsection~\ref{sec:exptight}. It relies on the exponential tightness of the sequence~$\big(P^{\otimes N}(L_N\in \cdot\,)\big)_{N\in\N}$ (cf. Definition~\ref{def:P}) and some estimates on the Radon-Nikodym derivative of $Q_{N}$ with respect to $P^{\otimes N}$.

\par In order to prove the weak LDP, we begin by providing bounds on the difference $|H_{\mu}(\mu)-H_{\nu}(\mu)|$, when $\mu$ and $\nu$ are close enough (see Lemma~\ref{lemma:rate_function}). Once this is done, the next step is to estimate the Radon-Nikodym derivative of $Q_{N}(L_{N} \in \cdot\,)$ in a neighborhood of a given measure $\nu \in \sM_{1}\left(\sC_{-\tau}^{T} \times \R^{d}\right)$ in terms of the Radon-Nikodym derivative of $Q_{\nu}^{\otimes N}(L_{N} \in \cdot \,)$ (see Lemma~\ref{lemma:estimates}). From these two main steps, we can conclude the weak LDP.

\par Theorem~\ref{thm:minimizer} has a shorter proof presented in Section~\ref{sec:minimizer}. The main idea is to observe that each minimizer $\mu\in\sM_1(\sC_{-\tau}^{T} \times \R^{d})$ of $H$ is a fixed point of the map $\nu \mapsto Q_{\nu}$ from Definition~\ref{def:decoupling}. In particular, the existence of minimizers implies the existence of path-dependent McKean-Vlasov diffusions. To conclude uniqueness, it suffices to verify that the map $\nu\mapsto Q_{\nu}$ has a unique fixed point, which will be a consequence of Banach's fixed-point Theorem. The map $\nu \mapsto Q_{\nu}$ is not by itself a contraction, but we are able to verify that a sufficiently large iteration of it is, implying the uniqueness in the statement. The PDE characterization is a simple consequence of It\^o's Formula.

\section{Fundamental Lemmas}\label{sec:fundamental_lemmas}
~
\par In this section we state the main lemmas we need to prove weak LDP and exponential tightness for the sequence $(Q_N(L_N\in \cdot\,))_{N \in \N}$. Recall the motivations and definitions presented in Subsection~\ref{sec:heuristics}.

\par In the first lemma, fixed $\nu\in \sM_1(\sC_{-\tau}^T\times \R^d)$, we provide a different expression for the rate function associated to the sequence $(Q_\nu^{\otimes N}(L_N\in \cdot\,))_{N \in \N}$. After this, we relate the new expression to the candidate rate function for the sequence $(Q_N(L_N\in \cdot\,))_{N \in \N}$.

\begin{lemma}[Rate-function Lemma]\label{lemma:rate_function} Let $\mu,\nu \in \M_{1}(\C_{-\tau}^{T} \times \R^{d})$. Then the following holds.
\begin{enumerate}
\item \label{prop:gamma_finite} If either $I(\mu|P)< \infty$ or $I(\mu|Q_\nu)< \infty$, then
\begin{equation*}
\int |D^{\nu,\omega}\big(x_{T}\big)| \,\d \mu(x_{T},\omega)<\infty.
\end{equation*}
Furthermore, there exists a positive constant $C>0$ depending on $T,$ $h_{*}$,  and $\pnorm{\infty}{f}$, such that
\begin{equation*}
|\Gamma_{\nu}(\mu)| \leq I(\mu|P)+C.
\end{equation*}
In particular,  $\Gamma_{\nu}(\mu)$ is finite whenever $I(\mu|P)$ is finite. Moreover, there exists $\delta \in (0,1)$ and $c>0$ such that
\begin{equation}\label{eq:bound_gamma}
\Gamma_{\nu}(\mu) \leq \delta I(\mu|P)+c.
\end{equation}
\item \label{lemma:ldp_perturbation}
The following equality holds:
\begin{equation*}
H_{\nu}(\mu)=I(\mu|Q_{\nu}).
\end{equation*}
\item \label{lemma:perturbation} There exists a constant $c>0$, depending only on $T,$ $h_{*}$,  and $\pnorm{\text{SL}}{f}$, such that
\begin{equation*}
|\Gamma_{\nu}(\mu)-\Gamma_{\mu}(\mu)| \leq c(I(\mu|P)+1)\d^T_{\text{BL}}(\mu,\nu).
\end{equation*}
In particular,
\[|H_\nu(\mu)-H_\mu(\mu)|\leq c(I(\mu|P)+1)\d^T_{\text{BL}}(\mu,\nu).\]
\end{enumerate}
\end{lemma}

\begin{remark}\label{remark:infinite_entropy}
In particular, the lemma above implies that $H_{\nu}(\mu)$ is finite if, and only if $I(\mu|P)$ is also finite.
\end{remark}

\par We prove the first item above in Subsection~\ref{proof:gamma_finite}. Item~\ref{lemma:ldp_perturbation} is proved in Subsection~\ref{proof:ldp_pertubation} and the proof of~\ref{lemma:perturbation} can be found in Subsection~\ref{proof:pertubation}.

\par The second lemma states useful estimates that we need when comparing $Q_N(L_N\in A)$ to either $Q_\nu^{\otimes N}(L_N \in A)$ or $P^{\otimes N}(L_N \in A)$, for a given measurable set $A\subset \sM_1(\sC_{-\tau}^T\times\R^d).$
\begin{lemma}[Fundamental estimates]\label{lemma:estimates}
Consider $\alpha \in \R$, $N\in \N$, and $\nu \in \sM_1(\sC_{-\tau}^T\times \R^d)$ fixed. 
\begin{enumerate}
\item \label{estimate:exptight}
There exists a positive constant $C$ depending on $h_*$, $\supnorm{f}$, and $T$ such that 
\begin{equation}\label{eq:l52-1}
\begin{split}
\int \bigg(\dfrac{\d Q_\nu^{\otimes N}}{\d P^{\otimes N}}\bigg)^{\alpha} \, \d P^{\otimes N} & \leq \exp\Big\{NC|\alpha^2-\alpha|\Big\}, \\
\int \bigg(\dfrac{\d Q_N}{\d P^{\otimes N}}\bigg)^{\alpha} \, \d P^{\otimes N} & \leq \exp\Big\{NC|\alpha^2-\alpha|\Big\}.
\end{split}
\end{equation}
\item \label{lemma:estimate_error_empirical_measure} 
There exists a positive constant $C$, depending on $\pnorm{\text{SL}}{f}$, $\supnorm{f}$, $h_*,$ and $T$, such that, for any $\eta \in (0,1)$,
\begin{equation}\label{eq:l52-2}
\int_{\{L_N \in B(\nu,\eta)\}}\bigg(\dfrac{\d Q_N}{\d Q_\nu^{\otimes N}} \bigg)^\alpha \, \d P^{\otimes N} \leq \exp\Big\{NC(\alpha^{2}+|\alpha|)\eta\Big\}.
\end{equation}
\item \label{item:RN2nd} There exists a positive constant $C$, depending on $\pnorm{BL}{f}$, $\supnorm{f}$, $h_*,$ and $T$, such that, for any $\eta \in (0,1)$,
\begin{equation}\label{eq:l52-3}
\int_{\{L_N \in B(\nu,\eta)\}}\bigg(\dfrac{\d Q_N}{\d Q_\nu^{\otimes N}} \bigg)^\alpha \, \d Q^{\otimes N}_{\nu}
\leq \exp \left\{NC\left(\alpha^2+|\alpha|)\sqrt{\eta})\right)\right\}.
\end{equation}
\end{enumerate}
\end{lemma}

\par We prove Item~\ref{estimate:exptight} of Lemma~\ref{lemma:estimates} in Subsection~\ref{proof:estimate}. Item~\ref{lemma:estimate_error_empirical_measure} of Lemma~\ref{lemma:estimates} is proved in Subsection~\ref{estimate:RN1st}. The last item is proved in Subsection~\ref{sec:RN2nd}.

\par In Section~\ref{sec:proof_wldp}, we use Lemmas~\ref{lemma:rate_function} and~\ref{lemma:estimates} to conclude the proof of the wLDP and exponential tightness. We use Lemma~\ref{lemma:rate_function} in Section~\ref{sec:minimizer} to conclude the proof of Theorem~\ref{thm:minimizer}. If the reader wishes, it is possible to assume the lemmas and skip directly to Section~\ref{sec:proof_wldp} or~\ref{sec:minimizer}.

\section{Rate-function Lemma}\label{subsec:rate_function}
~
\par In this section we present the proof of Lemma~\ref{lemma:rate_function}. We divide the proof in three subsections. We prove the first statement of the lemma in Subsection~\ref{proof:gamma_finite}. The proof of Item~\ref{lemma:ldp_perturbation} can be found in Subsection~\ref{proof:ldp_pertubation}. Finally, the proof of~\ref{lemma:perturbation} is presented in~\ref{proof:pertubation}.

\subsection{When $\Gamma_\nu(\mu)$ is finite}
\label{proof:gamma_finite}
~
\par We here prove Item~\ref{prop:gamma_finite} of Lemma~\ref{lemma:rate_function}. The idea of the proof is to use the entropy inequality~\eqref{eq:entropy_inequality} to relate the integrals with respect to $\mu$ and with respect to $P$.

\par Combining the entropy inequality with the inequality $e^{|x|} \leq e^{x}+e^{-x}$, we obtain
\begin{equation*}
\begin{split}
|\Gamma_{\nu}(\mu)| & \leq\int |D^{\nu, \omega}(x_{T})| \, \d \mu(x_{T},\omega) \\
& \leq I(\mu|P)+ \log \int e^{|D^{\nu, \omega}(x_{T})|} \, \d P(x_{T}, \omega) \\
& \leq I(\mu|P)+ \log \int \left( e^{D^{\nu, \omega}(x_{T})}+e^{-D^{\nu, \omega}(x_{T})} \right) \, \d P(x_{T}, \omega).
\end{split}
\end{equation*}

Similarly, recall~\eqref{eq:dQnu/dP} to obtain
\begin{equation*}
\begin{split}
\int |D^{\nu, \omega}|\,\d \mu & \leq I(\mu|Q_{\nu})+ \log \int \exp\{|D^{\nu, \omega}|\}\,\d Q_{\nu}\\
& = I(\mu|Q_{\nu})+ \log \int \exp\{|D^{\nu, \omega}|+D^{\nu, \omega}\}\,\d P\\
& \leq I(\mu|Q_{\nu})+ \log \int \exp\{2|D^{\nu, \omega}|\}\,\d P\\
& \leq I(\mu|Q_{\nu})+ \log \int \left(e^{2D^{\nu, \omega}}+e^{-2D^{\nu, \omega}}\right)\,\d P.
\end{split}
\end{equation*}

\par Therefore, assuming that either $I(\mu|P)<\infty$ or $I(\mu|Q_\nu)<\infty$, it suffices to verify that, for any $\alpha \in \R$,
\begin{equation}\label{eq:integral_bound}
\int e^{\alpha D^{\nu, \omega}(x_{T})} \, \d P(x_{T}, \omega) \leq \exp\left\{ \frac{|\alpha^{2}-\alpha|}{2h_{*}^{2}}T\pnorm{\infty}{f}^{2}\right\}.
\end{equation}

\par In order to prove~\eqref{eq:integral_bound}, we write $e^{\alpha D^{\nu, \omega}(x_{T})}$ as a product of a positive martingale and a bounded term. Notice first that Novikov's Condition (see~\cite[Corollary~3.5.13]{karatzas}) together with the fact that, under $P^{\omega}$, $\frac{1}{h(\omega)}\big(x(t)-x(0)\big)$ is a Brownian motion independent of $\big(x(t)\big)_{t \in [-\tau,0]}$ imply that, for each $\omega \in \R^{d}$,
\begin{equation}\label{eq:martingale}
M^{\omega, \alpha}(s) = \exp\left\{ \frac{\alpha}{h(\omega)^{2}}\int_{0}^{s}f(t, x_{t}, \nu, \omega)\,\d x(t)- \frac{\alpha^{2}}{2h(\omega)^{2}}\int_{0}^{s}f(t, x_{t}, \nu, \omega)^{2} \,\d t \right\},
\end{equation}
for $s \in [0,T]$, is a positive mean-one martingale with respect to $P^{\omega}$. Here, we use the independence between $\frac{1}{h(\omega)}\big(x(t)-x(0)\big)$ and $\big(x(t)\big)_{t \in [-\tau,0]}$ to obtain a martingale for each fixed realization of $\big(x(t)\big)_{t \in [-\tau,0]}$ and afterwards integrating with respect to the distribution of these functions. This type of argument is going to appear throughout the text and we will not mention this technicality anymore.

\par With this in mind, we rewrite $\alpha D^{\nu, \omega}(x_{T})$ as
\begin{align*}
\begin{split}
\alpha D^{\nu, \omega}(x_{T}) =& \frac{\alpha}{h(\omega)^{2}}\int_{0}^{T}f(t, x_{t}, \nu, \omega)\,\d x(t)- \frac{\alpha}{2h(\omega)^{2}}\int_{0}^{T}f(t, x_{t}, \nu, \omega)^{2} \,\d t  \\
=& \frac{\alpha}{h(\omega)^{2}}\int_{0}^{T}f(t, x_{t}, \nu, \omega)\,\d x(t)- \frac{\alpha^{2}}{2h(\omega)^{2}}\int_{0}^{T}f(t, x_{t}, \nu, \omega)^{2} \,\d t \\
 &+ \frac{\alpha^{2}-\alpha}{2h(\omega)^{2}}\int_{0}^{T}f(t, x_{t}, \nu, \omega)^{2} \,\d t .
\end{split}
\end{align*}

\par The equality above yields
\begin{equation*}
\begin{split}
\int e^{\alpha D^{\nu, \omega}(x_{T})} \, \d P^{\omega} &= \int M^{\omega, \alpha}(T)\exp\left\{ \frac{\alpha^{2}-\alpha}{2h(\omega)^{2}}\int_{0}^{T}f(t, x_{t}, \nu, \omega)^{2} \,\d t \right\} \, \d P^{\omega} \\
& \leq \exp\left\{ \frac{|\alpha^{2}-\alpha|}{2h_{*}^{2}}T\pnorm{\infty}{f}^{2}\right\} \int M^{\omega, \alpha}(T) \, \d P^{\omega} \\
& = \exp\left\{ \frac{|\alpha^{2}-\alpha|}{2h_{*}^{2}}T\pnorm{\infty}{f}^{2}\right\}.
\end{split}
\end{equation*}
Integrating the last expression with respect to $\mu_{\text{med}}$ concludes the proof of the first statement.

To verify~\eqref{eq:bound_gamma}, let $\alpha>1$ apply the entropy inequality~\eqref{eq:entropy_inequality} to $\alpha D^{\nu, \omega}(x_{T})$ and \eqref{eq:integral_bound} to obtain
\begin{equation*}
\begin{split}
\alpha \Gamma_{\nu}(\mu) & = \int \alpha D^{\nu, \omega}(x_{T}) \, \d \mu(x_{T},\omega) \\
& \leq I(\mu|P)+ \log \int e^{\alpha D^{\nu, \omega}(x_{T})} \, \d P(x_{T}, \omega) \\
& \leq I(\mu|P)+ \frac{|\alpha^{2}-\alpha|}{2h_{*}^{2}}T\pnorm{\infty}{f}^{2},
\end{split}
\end{equation*}
in view of~\eqref{eq:integral_bound}. Taking any $\alpha>1$ concludes the proof.

\subsection{An alternative expression for $I(\,\cdot\, |Q_\nu)$}\label{proof:ldp_pertubation}
~
\par The goal of this section is to prove Item~\ref{lemma:ldp_perturbation} of Lemma~\ref{lemma:rate_function}. We want to prove that, for any $\mu, \nu \in \sM_{1}(\sC_{-\tau}^{T} \times \R^{d})$,
\begin{equation}
H_{\nu}(\mu)=I(\mu|Q_{\nu}).
\end{equation}
We prove the equality in two steps:
\begin{equation}
I(\mu|Q_{\nu})\leq H_{\nu}(\mu) \mbox{ and }  H_{\nu}(\mu)\leq I(\mu|Q_{\nu}).
\end{equation}
Once again, the idea of the proof is to apply the entropy inequality~\eqref{eq:entropy_inequality} in several ways.

\textbf{First step:} $I(\mu|Q_{\nu})\leq H_{\nu}(\mu)$. We consider two cases: either $I(\mu|P)<\infty$ or $I(\mu|P)=\infty$.

\par If $I(\mu|P)=\infty$ then, from Definition~\ref{eq:def_of_Hnu}, we have $H_\nu(\mu)=\infty$ and the bound $I(\mu|Q_{\nu})\leq H_{\nu}(\mu)$ holds trivially.

\par Assume now that $I(\mu|P)<\infty$. In this case, we know from Item~\ref{prop:gamma_finite} of Lemma~\ref{lemma:rate_function} that
\begin{equation*}
\int \left| D^{\nu,\omega}\big(x_{T}\big) \right| \,\d \mu(x_{T},\omega)< \infty.
\end{equation*}
In particular, $\Gamma_{\nu}(\mu)$ is also finite.

\par Let $\phi:\C_{-\tau}^{T}\times \R^{d} \to \R$ be a bounded function and apply the entropy inequality~\eqref{eq:entropy_inequality} to obtain
\begin{equation*}
\begin{split}
\int \phi\,\d \mu &= \int (\phi+D^{\nu, \omega}) \,\d \mu - \Gamma_{\nu}(\mu) \\
& \leq I(\mu|P) + \log \int \exp\{\phi+D^{\nu, \omega}\}\,\d P - \Gamma_{\nu}(\mu) \\
& =I(\mu|P) - \Gamma_{\nu}(\mu) + \log \int \exp\{\phi\}\,\d Q_{\nu},
\end{split}
\end{equation*}
which yields
\begin{equation*}
\int \phi\,\d \mu-\log \int \exp\{\phi\}\,\d Q_{\nu}\leq I(\mu|P) - \Gamma_{\nu}(\mu).
\end{equation*}
Taking the supremum over all possible choices of $\phi$, we have
\begin{equation*}
I(\mu|Q_{\nu}) \leq I(\mu|P) - \Gamma_{\nu}(\mu) = H_{\nu}(\mu),
\end{equation*}
concluding the first step.

\textbf{Second step:} $H_{\nu}(\mu)\leq I(\mu|Q_{\nu})$. Notice that if $I(\mu|Q_{\nu})$ is infinite, the inequality holds trivially. Hence, we assume that $I(\mu|Q_{\nu})$ is finite. From Item~\ref{prop:gamma_finite} of Lemma~\ref{lemma:rate_function}, $D^{\nu,\omega}$ is $\mu-$integrable. 

Let us show that $I(\mu|P)$ is finite. For any bounded function $\phi: \C_{-\tau}^{T} \times \R^{d} \to \R$, using~\eqref{eq:dQnu/dP},
\begin{align*}
\int \phi\,\d\mu &=\int (\phi-D^{\nu, \omega})\,\d\mu+\int D^{\nu, \omega}\,\d \mu\\
&\leq I(\mu\vert Q_\nu)+\log \int \exp\{\phi-D^{\nu, \omega}\}\,\d Q_\nu+\int D^{\nu, \omega}\,\d \mu\\
&\leq I(\mu\vert Q_\nu)+\log \int \exp\{\phi\}\,\d P+\int D^{\nu, \omega}\,\d \mu.
\end{align*}
As consequence,
\begin{align*}
\int \phi\,\d\mu-\log \int \exp\{\phi\}\,\d P\leq I(\mu\vert Q_\nu) +\int D^{\nu, \omega}\,\d \mu.
\end{align*}
Taking the supremum over all possible choices of $\phi$,
\begin{equation*}
I(\mu|P) \leq I(\mu|Q_{\nu}) + \int D^{\nu, \omega}\,\d \mu < \infty,
\end{equation*}
and therefore
\begin{equation*}
H_\nu(\mu)=I(\mu|P)-\int D^{\nu,\omega}\,\d \mu\leq I(\mu|Q_\nu),
\end{equation*}
concluding the proof.

\subsection{Comparing $\Gamma_\nu(\mu)$ with $\Gamma_\mu(\mu)$}
\label{proof:pertubation}
~
\par In this subsection, we prove Item~\ref{lemma:perturbation} of Lemma~\ref{lemma:rate_function}.

\par We can assume that $I(\mu|P)<\infty$ since otherwise the bound is trivial. Under this assumption we also know from Item~\ref{prop:gamma_finite} of Lemma~\ref{lemma:rate_function} that
\begin{equation*}
\Gamma_\nu(\mu)=\int D^{\nu,\omega}\big(x_{T}\big) \, \d \mu(x_{T},\omega)
\end{equation*}
is finite. Therefore, we need to verify the bound 
\begin{equation}\label{goal:D}
\left| \int \big[ D^{\nu,\omega}\big(x_{T}\big)-D^{\mu,\omega}\big(x_{T}\big)\big] \, \d \mu(x_{T},\omega) \right| \leq c(I(\mu|P)+1)\d^T_{\text{BL}}(\mu,\nu),
\end{equation} 
where $D^{\nu, \omega}$ is given in~\eqref{def:D} as
\begin{equation}
D^{\nu,\omega}(x_{T}):=\frac{1}{h(\omega)^{2}}\int_{0}^{T}f(t,x_{t}, \nu, \omega)\,\d x(t)- \frac{1}{2h(\omega)^{2}}\int_{0}^{T}f(t, x_{t}, \nu, \omega)^{2} \,\d t.
\end{equation}

\par We consider the stochastic integral separately from the usual integral and write
\begin{equation}\label{eq:perturbation_1}
\int \big[ D^{\nu,\omega}\big(x_{T}\big)-D^{\mu,\omega}\big(x_{T}\big)\big]\,\d \mu(x_{T},\omega) =\int \big[B_1(x_T,\omega)+B_2(x_T,\omega)\big]\, \d \mu(x_T,\omega)
\end{equation}
where, for each pair $(x_{T}, \omega) \in \sC_{-\tau}^{T} \times \R^{d}$,
\begin{equation}\label{eq:perturbation_B1}
B_{1}(x_T,\omega) := \frac{1}{2h(\omega)^{2}}\int_{0}^{T}\big[f(t,x_{t}, \omega, \nu)^{2}-f(t,x_{t}, \omega, \mu)^{2}\big] \, \d t,
\end{equation}
and
\begin{equation}\label{eq:perturbation_B2}
B_{2}(x_T,\omega) :=\frac{1}{h(\omega)^{2}}\int_{0}^{T} \big[f(t,x_{t}, \omega, \nu)-f(t,x_{t}, \omega, \mu) \big] \, \d x(t).
\end{equation}

We will prove that
\begin{align}
\int |B_{1}(x_T,\omega)| \, \d \mu(x_T,\omega)&\leq \frac{1}{h_{*}^{2}} \pnorm{\infty}{f}\pnorm{\text{SL}}{f} T \d_{\text{BL}}^{T}(\mu,\nu),\label{bound:B1}\\
\int |B_{2}(x_T,\omega)| \, \d \mu(x_T,\omega) &\leq \frac{1}{h_{*}}\pnorm{\text{SL}}{f}T^{\frac{1}{2}}\left(I(\mu|P)+1\right)\d^T_{\text{BL}}(\mu,\nu).\label{bound:B2}
\end{align}
Once we verify the previous bounds the proof is completed with the evident choice of $c$.

We proceed to prove the bound in~\eqref{bound:B1}. For this, recall our assumptions on $f$ (cf. Section~\ref{subsec:assumptions}). Using the identity $a^{2}-b^{2}=(a-b)(a+b)$ we obtain, uniformly in $(x_{T},\omega) \in \C_{-\tau}^{T} \times \R^{d}$,
\begin{multline}\label{eq:perturbation_2}
\begin{split}
|B_{1}(x_T,\omega)| &\leq  \\
\leq\frac{1}{2h(\omega)^{2}}&\int_{0}^{T}\big|f(t,x_{t},\nu,\omega)+f(t,x_{t},\mu, \omega)\big|\big|f(t,x_{t},\nu, \omega)-f(t,x_{t},  \mu,\omega)\big| \, \d t,\\
& \leq \frac{1}{h_{*}^{2}} \pnorm{\infty}{f}\int_{0}^{T}\left| f(t,x_{t}, \nu, \omega)-f(t,x_{t}, \mu,\omega)\right|\,\d t \\
& \leq \frac{1}{h_{*}^{2}} \pnorm{\infty}{f}\pnorm{\text{SL}}{f}\int_{0}^{T}\d_{\text{BL}}^{t}(\nu , \mu )\,\d t  \leq \frac{1}{h_{*}^{2}} \pnorm{\infty}{f}\pnorm{\text{SL}}{f} T \d_{\text{BL}}^{T}(\mu,\nu),
\end{split}
\end{multline}
which concludes the proof of~\eqref{bound:B1}.

We now prove the bound in~\eqref{bound:B2}. We cannot apply the entropy inequality~\eqref{eq:entropy_inequality} directly to $\int |B_{2}(x_T,\omega)| \, \d \mu(x_T,\omega)$, because it would yield an expression with a free factor $I(\mu |P)$ that needs to be multiplied by $\d^T_{\text{BL}}(\mu,\nu)$. The trick here is to apply the entropy inequality for $B_2(x_T,\omega)/\beta,$ where $\beta$ is a constant conveniently chosen.

We write
\begin{align*}
\int |B_{2}(x_T,\omega)| \, \d \mu(x_T,\omega)&=\beta\int \frac{|B_{2}(x_T,\omega)|}{\beta}\,\d \mu(x_T,\omega)
\end{align*}
and, by choosing $\beta = \frac{1}{h_{*}}\pnorm{\text{BL}}{f}T^{1/2}\d_{\text{BL}}^{T}(\mu,\nu)$, it suffices to verify that
\begin{align}\label{bound:entropy}
\int \frac{|B_{2}(x_T,\omega)|}{\beta}\,\d \mu(x_T,\omega)\leq  I(\mu|P)+1.
\end{align}
Applying the entropy inequality~\eqref{eq:entropy_inequality}
\begin{equation}\label{beta-entropy}
\int \frac{|B_{2}(x_T,\omega)|}{\beta}\,\d \mu(x_T,\omega)\leq I(\mu | P)+\log \int \exp\left\{\frac{|B_{2}(x_T,\omega)|}{\beta}\right\}\,\d P(x_T,\omega).
\end{equation}

We now need to estimate the second term of the RHS of the previous equation. For that, we will use a suitable martingale. From Novikov's Condition, for every $\alpha \in \R$ and $\omega \in \R^{d}$, 
\begin{multline*}
M^{\alpha, \nu, \mu,\omega}(x_T,s) =\exp \Bigg\{  \int_{0}^{s}\frac{\alpha}{h(\omega)^{2}}\big[ f(t,x_{t}, \nu, \omega) - f(t,x_{t}, \mu, \omega)\big]\,\d x(t) \\
-\int_{0}^{s}\frac{\alpha^{2}}{2h(\omega)^2}\big[ f(t,x_{t}, \nu, \omega) - f(t,x_{t}, \mu, \omega) \big]^{2} \,\d t\Bigg\},
\end{multline*}
for $s\in [0,T],$ is a positive mean-one martingale with respect to $ P^{\omega}$ (see also Equation~\ref{eq:martingale}).

Therefore, summing and subtracting the suitable term
\begin{multline*}
\exp\left\{\alpha \frac{B_{2}(x_T,\omega)}{\beta}\right\}= M^{\alpha/\beta, \nu, \mu,\omega}(\vec x_T,s)\\
\times\exp\left\{\frac{\alpha^2}{2\beta^2h(\omega)^2}\int_{0}^{T} \big[f(t,x_{t}, \nu,\omega)-f(t,x_{t}, \mu,\omega) \big]^2 \, \d t\right\}.
\end{multline*}

Proceeding as in Equation~\eqref{eq:perturbation_2}, we have the following estimate, uniformly in $\omega \in \R^{d}$,  for the quadratic variation at time $T$ 
\begin{equation}
\frac{1}{h(\omega)^{2}} \int_{0}^{T} \Big( f(t,x_{t},\nu,\omega)-f(t,x_{t}, \mu,\omega) \Big)^{2}\,\d t \leq \frac{1}{h_{*}^{2}}\pnorm{\text{SL}}{f}^2T\d_{\text{BL}}^{T}(\mu,\nu)^2
\end{equation}
and, therefore,
\begin{equation*}
\log \int \exp\Big\{\alpha \frac{B_{2}(x_T,\omega)}{\beta}\Big\}\,\d P^\omega(x_T)\leq \frac{\alpha^2}{2\beta^2h_{*}^{2}}\pnorm{\text{SL}}{f}^2T\d_{\text{BL}}^{T}(\mu,\nu)^2.
\end{equation*}
As consequence, using the inequality $e^{|x|}\leq e^x+e^{-x},$
\begin{equation*}
\log \int \exp\Bigg\{\frac{|B_{2}(x_T,\omega)|}{\beta}\Bigg\}\,\d P(x_T,\omega)\leq \frac{1}{\beta^2h_{*}^{2}}\pnorm{\text{SL}}{f}^2T\d_{\text{BL}}^{T}(\mu,\nu)^2.
\end{equation*}
The choice of $\beta$ together with~\eqref{beta-entropy} concludes the proof of~\eqref{bound:entropy}.

\section{Fundamental Estimates}\label{sec:fundamental_estimates}
~
\par This section contains the proof of Lemma~\ref{lemma:estimates}. We prove Item~\ref{estimate:exptight} in Subsection~\ref{proof:estimate}. The proof of Item~\ref{lemma:estimate_error_empirical_measure} can be found in Subsection~\ref{estimate:RN1st}, while the proof of Item~\ref{item:RN2nd} is presented in Subsection~\ref{sec:RN2nd}.

\subsection{The moment generating function}
\label{proof:estimate}
~
\par In this subsection we prove the estimates in Item~\ref{estimate:exptight} of Lemma~\ref{lemma:estimates}.

Observe that the first estimate can be deduced from Estimate~\eqref{eq:integral_bound}  since
\begin{equation*}
\alpha N\int D^{\nu,\sigma}(y_T) \, L_N(\d y_T,\d \sigma)=\sum_{i=1}^N \alpha D^{\nu,\omega^i}(x^i_T)
\end{equation*}
and we are integrating over a product measure. Therefore, using~\eqref{eq:dQnu/dP}, we obtain the bound
\begin{equation*}
\int \bigg(\dfrac{\d Q_\nu^{\otimes N}}{\d P^{\otimes N}}\bigg)^{\alpha} \, \d P^{\otimes N} \leq \exp\Big\{NC_1|\alpha^2-\alpha|\Big\},
\end{equation*}
for a positive constant $C_1$ depending on $\supnorm{f},$ $h_*$, and $T$.

We now deal with the second estimate. From~\eqref{eq:dQN/dP} our goal is to obtain bounds for the moment generating function
\begin{equation}\label{goal:estimate}
\int \exp\left\{\alpha N\int D^{L_N(\vec{x}_T,\vec{\omega}),\sigma}\big(y_T\big) \, L_N(\vec{x}_T,\vec{\omega})(\d y_T,\d \sigma) \right\}\, \d P^{\otimes N}(\vec{x}_T,\vec{\omega}).
\end{equation}
In the following, we omit the dependence of $L_N$ on the vector $(\vec{x}_T,\vec{\omega})$.

\par Recall from~\eqref{def:D} that
\[D^{L_N,\sigma}(y_{T}):=\frac{1}{h(\sigma)^{2}}\int_{0}^{T}f(t,y_{t}, L_N, \sigma)\,\d y(t)- \frac{1}{2h(\sigma)^{2}}\int_{0}^{T}f(t, y_{t}, L_N, \sigma)^{2} \,\d t.
\]

\par We will apply a similar strategy as in the proof of Estimate~\eqref{eq:integral_bound}. The idea is to rewrite~\eqref{goal:estimate} as a product of a positive martingale and a bounded term.

\par Given a vector $(\vec{x}_T,\vec{\omega})$, we have $L_{N}= \frac{1}{N}\sum_{i=1}^{N} \delta_{(x^{i}_{T}, \omega^{i})}$, and directly from the definition of $L_N$ and $D^{L_N,\omega}$, we obtain
\begin{equation*}
\begin{split}
& \alpha N\int D^{L_N,\sigma}\big(y_T\big) \, L_N(\d y_T,\d \sigma) = \alpha \sum_{i=1}^{N} D^{L_N,\omega^i}(x_T^i)  \\
& =  \sum_{i=1}^{N} \left[ \frac{\alpha}{h(\omega^i)^{2}}\int_{0}^{T}f(t,x^i_t, L_N, \omega^i )\,\d x^i(t) - \frac{\alpha}{2h(\omega^i)^{2}} \int_{0}^{T}f(t, x^i_t, L_N, \omega^i)^{2} \,\d t \right].
\end{split}
\end{equation*}

Notice that, for every $\alpha \in \R$, $N\in \N,$ $ \vec\omega \in \R^{Nd}$, 
\begin{multline*}
M^{\alpha,N,\vec\omega}(\vec x_T,s)  = \exp \Bigg\{  \sum_{i=1}^{N} \bigg[ \frac{\alpha}{h(\omega^i)^{2}}\int_{0}^{s}f(t,x^i_t, L_N, \omega^i)\,\d x(t)\\
- \frac{\alpha^2}{2h(\omega^i)^{2}} \int_{0}^{s}f(t, x^i_t, L_N, \omega^i)^{2} \,\d t \bigg] \Bigg\},
\end{multline*}
for $s\in [0,T],$ is a mean-one martingale with respect to $\prod_{i=1}^{N} P^{\omega^{i}}$ (see also Equation~\ref{eq:martingale}).

With this is mind, we write
\begin{multline*}
\int \exp\left\{\alpha N\int D^{L_N,\sigma}\big(y_T\big) \, L_N(\d y_{T},\d \sigma) \right\}\, \d P^{\otimes N}(\vec{x}_T,\vec{\omega})=\\
 =\int M^{\alpha,N,\vec\omega}\exp \left\{  \sum_{i=1}^{N} \frac{\alpha^2-\alpha}{2h(\omega^i)^{2}} \int_{0}^{T} f(t, x^i_t, L_N, \omega^i)^{2}  \,\d t \right\} \prod_{i=1}^{N}
 \, \d P^{\omega^{i}} \, \d \mu_{\text{med}}.
\end{multline*}
To conclude, we just need to plug the uniform estimate (cf. Subsection~\ref{subsec:assumptions})
\begin{align*}
\sum_{i=1}^{N} \frac{\alpha^2-\alpha}{2h(\omega^i)^{2}} \int_{0}^{T} f(t, x^i_t, L_N, \omega^i)^{2}  \,\d t\leq N\frac{|\alpha^2-\alpha|}{2h_*^{2}} T\supnorm{ f}^2,
\end{align*}
and use that the martingale has mean one.

\subsection{The Radon-Nikodym derivatives - Part I}
\label{estimate:RN1st}
~
\par In this subsection, we prove Item~\ref{lemma:estimate_error_empirical_measure} of Lemma~\ref{lemma:estimates}. Our goal is to estimate
\begin{multline}\label{goal:RN1st}
\int_{E_\eta(\nu)}\exp\Bigg\{\alpha N \int \big(D^{L_{N}, \sigma}(y_T)- D^{\nu, \sigma}(y_T)\big) \,L_N(\d y_T,\d \sigma) \Bigg\} \, \d P^{\otimes N}(\vec x_T,\vec \omega),
\end{multline}
where $E_\eta(\nu)$ is the event $\{L_N \in B(\nu,\eta)\}$. Notice that in the equation above we once again omitted the dependence of $L_{N}$ on the vector $(\vec x_T,\vec \omega)$.

\par In order to bound the quantity above, we will combine the ideas used in~\eqref{eq:integral_bound} and Item~\ref{estimate:exptight} of Lemma~\ref{lemma:estimates} (see Equation~\eqref{eq:l52-1}). We will expand the expression and decompose it as a product of a martingale and a bounded term.

Recalling from~\eqref{def:D} that
\[D^{L_N,\sigma}(y_{T}):=\frac{1}{h(\sigma)^2}\int_{0}^{T}f(t,y_{t}, L_N, \sigma)\,\d y(t)- \frac{1}{2h(\sigma)^{2}}\int_{0}^{T}f(t, y_{t}, L_N, \sigma)^{2} \,\d t,
\]
we write
\begin{multline*}
\alpha N \int \big[D^{L_{N}, \sigma}(y_T)- D^{\nu, \sigma}(y_T)\big] \,L_N(\d y_T,\d \sigma)=\\
=\sum_{i=1}^{N} \int_{0}^{T} \frac{\alpha}{h(\omega^i)^2}\big[f(t,x_{t}^{i}, L_{N}, \omega^{i}) - f(t,x_{t}^{i}, \nu, \omega^{i})\big]\,\d x^{i}(t)\\
-\frac{\alpha}{2h(\omega^i)^2}\sum_{i=1}^{N}\int_{0}^{T}\big[ f(t,x_{t}^{i}, L_{N}, \omega^{i})^2 - f(t,x_{t}^{i}, \nu, \omega^{i})^2 \big] \,\d t.
\end{multline*}
Motivated by the previous expression, we notice that, for every $\alpha \in \R$, $N\in \N,$ $ \vec\omega \in \R^{Nd}$, 
\begin{multline}\label{eq:martingale_radon_nykodin}
M^{\alpha, \nu, N,\vec\omega}(\vec x_T,s) =\exp \Bigg\{ \sum_{i=1}^{N} \int_{0}^{s}\frac{\alpha}{h(\omega^i)^2}\big[ f(t,x_{t}^{i}, L_{N}, \omega^{i}) - f(t,x_{t}^{i}, \nu, \omega^{i})\big]\,\d x^{i}(t) \\
-\sum_{i=1}^{N}\int_{0}^{s}\frac{\alpha^{2}}{2h(\omega^i)^2}\big[ f(t,x_{t}^{i}, L_{N}, \omega^{i}) - f(t,x_{t}^{i}, \nu, \omega^{i}) \big]^{2} \,\d t\Bigg\},
\end{multline}
for $s\in [0,T],$ is a positive mean-one martingale with respect to $\prod_{i=1}^{N} P^{\omega_{i}}$ (see Equation~\ref{eq:martingale}).

Adding and subtracting the correct quantity, we obtain that
\begin{multline*}
\exp\Bigg\{\alpha N \int \big[D^{L_{N}, \sigma}(y_T)- D^{\nu, \sigma}(y_T)\big] \,L_N(\d y_T,\d \sigma) \Bigg\} =\\
= M^{\alpha, \nu, N,\vec\omega}(\vec x_T,T)\exp\Big\{B_1^\alpha(\vec{x}_T,\vec{\omega})+B_2^\alpha(\vec{x}_T,\vec{\omega})\Big\},
\end{multline*}
where
\begin{align*}
B_{1}^\alpha(\vec{x}_T,\vec{\omega}) & :=\sum_{i=1}^{N}\int_{0}^{T}\frac{\alpha^{2}}{2h(\omega^i)^2}\big[f(t,x_{t}^{i}, L_{N}, \omega^{i}) - f(t,x_{t}^{i}, \nu, \omega^{i})\big]^{2} \,\d t,\\
B_{2}^\alpha(\vec{x}_T,\vec{\omega})&:=-\sum_{i=1}^{N}\int_{0}^{T}\frac{\alpha}{2h(\omega^i)^2}\big[f(t,x_{t}^{i}, L_{N}, \omega^{i})^{2} - f(t,x_{t}^{i}, \nu, \omega^{i})^{2}\big] \,\d t.
\end{align*}

\par Assume for a moment the following bounds, proved in the end of this subsection,
\begin{align}
B_1^\alpha(\vec{x}_T,\vec{\omega})&\leq \frac{\alpha^{2}}{2h_*^2}N\pnorm{\text{SL}}{f}^{2}T\d^{T}_{\text{BL}}(L_{N}, \nu)^{2},\label{B_1}\\
B_2^\alpha(\vec{x}_T,\vec{\omega})&\leq \frac{|\alpha|}{h_*^2} N \pnorm{\infty}{f}\pnorm{\text{SL}}{f} T\d^{T}_{\text{BL}}(L_{N}, \nu).\label{B_2}
\end{align}
In the following, using that the martingale $M^{\alpha, \nu, N}$ is positive, we obtain
\begin{multline*}
\int_{E_\eta(\nu)}M^{\alpha, \nu, N,\vec\omega}(T) \,\d P^{\otimes N}(\vec x_T,\vec \omega)\\
\leq \int M^{\alpha, \nu, N,\vec\omega}(\vec x_T,T) \prod_{i=1}^{N} \, \d P^{\omega^i}(x^i_T) \, \d\mu_{\text{med}}(\omega^i)=1.
\end{multline*}
Since we are integrating on the event $E_\eta(\nu)=\{L_N \in B(\nu,\eta)\}$, we have
\begin{align*}
\int_{E_\eta(\nu)}&\exp\Bigg\{\alpha N \int \left(D^{L_{N}, \sigma}(y_T)- D^{\nu, \sigma}(y_T)\right) \,L_N(\d y_T,\d \sigma) \Bigg\} \, \d P^{\otimes N}(\vec x_T,\vec \omega)\\
&\leq \int_{E_\eta(\nu)} M^{\alpha, \nu, N,\vec\omega}(\vec x_T,T) \\
& \qquad \qquad \qquad \times \exp\Bigg\{\frac{\alpha^{2}}{2h_*^2}N\pnorm{\text{SL}}{f}^{2}T\eta^2+\frac{|\alpha|}{h_*^2} N \pnorm{\infty}{f}\pnorm{\text{SL}}{f} T\eta\Bigg\}\d P^{\otimes N}\\
&\leq \exp\bigg\{\frac{\alpha^{2}}{2h_*^2}N\pnorm{\text{SL}}{f}^{2}T\eta^2+\frac{|\alpha|}{h_*^2} N \pnorm{\infty}{f}\pnorm{\text{SL}}{f} T\eta\bigg\}\\
&\leq \exp\Big\{NC(\alpha^{2}+|\alpha|)\eta\Big\},
\end{align*}
for a constant $C$ depending on $\pnorm{\text{SL}}{f}$, $\supnorm{f},$ $h_*,$ and $T$.

\par We dedicate the rest of this section to prove the bounds in~\eqref{B_1} and~\eqref{B_2}. Recall the assumption on the functions $f$ and $h$ made in Subsection~\ref{subsec:assumptions} and estimate
\begin{equation*}
\begin{split}
B_{1}^\alpha(\vec{x}_T,\vec{\omega}) & :=\sum_{i=1}^{N}\int_{0}^{T}\frac{\alpha^{2}}{2h(\omega^i)^2}\big[f(t,x_{t}^{i}, L_{N}, \omega^{i}) - f(t,x_{t}^{i}, \nu, \omega^{i})\big]^{2} \,\d t \\
& \leq \frac{\alpha^{2}}{2h_*^2}N\pnorm{\text{SL}}{f}^{2} \int_{0}^{T}\d^{t}_{\text{BL}}(L_{N} , \nu )^{2}\,\d t \\
& \leq \frac{\alpha^{2}}{2h_*^2}N\pnorm{\text{SL}}{f}^{2}T\d^{T}_{\text{BL}}(L_{N}, \nu)^{2}.
\end{split}
\end{equation*}

\par Similarly, with the identity $a^{2}-b^{2}=(a-b)(a+b)$, we obtain
\begin{equation*}
\begin{split}
B_{2}^\alpha(\vec{x}_T,\vec{\omega})&=-\sum_{i=1}^{N}\int_{0}^{T}\frac{\alpha}{2h(\omega^i)^2}\big[f(t,x_{t}^{i}, L_{N}, \omega^{i})^{2} - f(t,x_{t}^{i}, \nu, \omega^{i})^{2}\big] \,\d t \\
& \leq \frac{|\alpha|}{h_*^2} \pnorm{\infty}{f}\sum_{i=1}^{N}\int_{0}^{T}\left|f(t,x_{t}^{i}, L_{N}, \omega^{i})-f(t,x_{t}^{i}, \nu, \omega^{i})\right| \,\d t \\
& \leq \frac{|\alpha|}{h_*^2} \pnorm{\infty}{f}N\pnorm{\text{SL}}{f}T\d^T_{\text{BL}}(L_N,\nu).
\end{split}
\end{equation*}

\par This verifies~\eqref{B_1} and~\eqref{B_2} and concludes the proof of Item~\ref{lemma:estimate_error_empirical_measure} of Lemma~\ref{lemma:estimates}.

\subsection{The Radon-Nikodym derivatives - Part II}
\label{sec:RN2nd}
~
\par We now proceed to the proof of the last item of Lemma~\ref{lemma:estimates}. Observe that
\begin{align*}
\int_{\{L_N \in B(\nu,\eta)\}}\bigg(\dfrac{\d Q_N}{\d Q_\nu^{\otimes N}} \bigg)^\alpha \, \d Q^{\otimes N}_{\nu}=\int_{\{L_N \in B(\nu,\eta)\}}\bigg(\dfrac{\d Q_N}{\d Q_\nu^{\otimes N}} \bigg)^\alpha  \dfrac{\d Q^{\otimes N}_{\nu}}{\d P^{\otimes N}} \, \d P^{\otimes N}. 
\end{align*}
Fix $r,s \in (1,\infty)$, with $1/r+1/s=1$. Applying H\"{o}lder's Inequality, we bound the integral above by
\begin{multline*}
\Bigg[\int_{\{L_N \in B(\nu,\eta)\}}\bigg(\dfrac{\d Q_N}{\d Q_\nu^{\otimes N}} \bigg)^{r\alpha} \, \d P^{\otimes N}\Bigg]^{1/r}
\Bigg[\int_{\{L_N \in B(\nu,\eta)\}}  \bigg(\dfrac{\d Q^{\otimes N}_{\nu}}{\d P^{\otimes N}}\bigg)^s \, \d P^{\otimes N}\Bigg]^{1/s}.
\end{multline*}
\par Items~\ref{estimate:exptight} and~\ref{lemma:estimate_error_empirical_measure} of Lemma~\ref{lemma:estimates} (see Equations~\eqref{eq:l52-1} and~\eqref{eq:l52-2}) yield
\begin{align*}
\int_{\{L_N \in B(\nu,\eta)\}}  \bigg(\dfrac{\d Q^{\otimes N}_{\nu}}{\d P^{\otimes N}}\bigg)^s \, \d P^{\otimes N}&\leq \exp\Big\{NC_1|s^2-s|\Big\},\\
\int_{\{L_N \in B(\nu,\eta)\}}\bigg(\dfrac{\d Q_N}{\d Q_\nu^{\otimes N}} \bigg)^{r\alpha} \, \d P^{\otimes N}&\leq \exp\Big\{NC_2(r^2\alpha^2+r|\alpha|)\eta\Big\},
\end{align*}
where $C_1$ is a positive constant depending on $\supnorm{f},$ $h_*$, and $T$ and $C_2$ is positive constant depending on $\pnorm{\text{SL}}{f}$, $\supnorm{f}$, $h_*,$ and $T$.

\par Combining these bounds, taking $C=C_1+C_2$, and using that $1/r+1/s=1$, we obtain
\begin{align*}
\int_{\{L_N \in B(\nu,\eta)\}}\bigg(\dfrac{\d Q_N}{\d Q_\nu^{\otimes N}} \bigg)^\alpha \, \d Q^{\otimes N}_{\nu} &\leq \exp\bigg\{ NC\big(r\alpha^2\eta+|\alpha|\eta+|s-1|\big)\bigg\}\\
&=\exp\bigg\{ NC\Big(r\alpha^2\eta+|\alpha|\eta+\frac{1}{r-1}\Big)\bigg\}.
\end{align*}
Optimizing over $r>1$ we obtain an optimal bound with $r=1+\frac{1}{|\alpha|\sqrt{\eta}}$ and this concludes the proof of Item~\ref{item:RN2nd} of Lemma~\ref{lemma:estimates}.

\section{Exponential tightness and weak LDP}\label{sec:proof_wldp}
~
\par In this section, we will prove that the sequence $Q_N(L_N \in \cdot\,)$ is exponentially tight and satisfies a weak LDP with rate function $H$. This is equivalent to the statement of Theorem~\ref{t:ldp}.

\par The lower bound for open sets is proved in Subsection~\ref{lower}. We will prove the upper bound on compact sets in Subection~\ref{subsec:upper_bound}. In Subsection~\ref{sec:exptight} we prove exponential tightness.

\subsection{Lower bound for open sets}\label{lower}
~
\par In this subsection we will prove the lower bound for open sets. More specifically, for any open set $\sO \subset \sM_1\big(\sC_{-\tau}^{T} \times \R^{d}\big)$, we verify that
\begin{equation*}
\liminf_{N\to\infty} \frac{1}{N} \log Q_{N}( L_{N} \in \sO) \geq -\inf_{\mu \in \sO} H(\mu).
\end{equation*}
Observe that the bound above can easily be deduced if we conclude that, for all $\nu \in \sO$,
\begin{equation}\label{lower:Q}
\liminf_{N\to\infty} \frac{1}{N}\log Q_{N}(L_{N} \in \sO)\geq -H(\nu).
\end{equation}

\par The idea of the proof is the following. Let $\nu \in \sO$ and $\eta>0$ such that $B(\nu,\eta)\subset \sO.$ Sanov's Theorem applied to the sequence $Q_\nu^{\otimes N}(L_N \in \cdot\,)$ (cf. Section~\ref{sec:heuristics}) gives
\begin{equation}\label{lower:Qnu}
\liminf_{N\to\infty} \frac{1}{N} \log Q_{\nu}^{\otimes N}( L_{N} \in B(\nu,\eta)) \geq -\inf_{\mu \in B(\nu,\eta)} H_\nu(\mu)\geq -H_\nu(\nu).
\end{equation}
By definition, $H_\nu(\nu)=H(\nu)$. We then just need to relate the probabilities that appear in~\eqref{lower:Q} and~\eqref{lower:Qnu}. Let us now provide the details.

\par Fix $\nu \in \sO$ and $\eta_0>0$ such that $B(\nu,\eta_0)\subset \sO.$ Let $\eta \in (0,\eta_0)$ and $a,\,b\in (1,+\infty)$ with $1/a+1/b=1$. By applying H\"older's Inequality, we obtain
\begin{multline}\label{open:bound1}
Q^{\otimes N}_{\nu}(L_{N} \in B(\nu, \eta))=\int_{\{L_{N} \in B(\nu, \eta)\}}\dfrac{\d Q_\nu^{\otimes N}}{\d Q_N} \, \d Q_N\\
\leq Q_N(L_{N} \in B(\nu, \eta))^{1/a}\bigg[\int_{\{L_{N} \in B(\nu, \eta)\}}\bigg(\dfrac{\d Q_\nu^{\otimes N}}{\d Q_N}\bigg)^b \, \d Q_N\bigg]^{1/b}.
\end{multline}

\par By Item~\ref{item:RN2nd} of Lemma~\ref{lemma:estimates} (with $\alpha=b-1$), there exists a constant $C=C(\pnorm{\text{SL}}{f},h_*)$ such that, for any $\eta \in (0,\min\{\eta_0,1\})$,
\begin{multline*}
\bigg[\int_{\{L_{N} \in B(\nu, \eta)\}}\bigg(\dfrac{\d Q_\nu^{\otimes N}}{\d Q_N}\bigg)^b \, \d Q_N\bigg]^{1/b}=\bigg[\int_{\{L_{N} \in B(\nu, \eta)\}}\bigg(\dfrac{\d Q_\nu^{\otimes N}}{\d Q_N}\bigg)^{b-1} \, \d Q_\nu^{\otimes N}\bigg]^{1/b}\\
\leq \exp\left\{NC\frac{1}{b}((b-1)^2+b-1)\sqrt{\eta}\right\}\\
= \exp\left\{NC(b-1)\sqrt{\eta}\right\}.
\end{multline*}

\par Plugging this bound in~\eqref{open:bound1} and using that $1/a=1-1/b$, we obtain
\begin{multline*}
\dfrac{1}{N}\log Q^{\otimes N}_{\nu}(L_{N} \in B(\nu, \eta))\leq (1-1/b)\dfrac{1}{N}\log Q_N(L_{N} \in B(\nu, \eta))\\
+C(b-1)\sqrt{\eta}. 
\end{multline*}

\par Taking the inferior limit as $N\to\infty$ in the above, using the lower bound~\eqref{lower:Qnu} and $B(\nu,\eta)\subset \sO$ we obtain 
\begin{align*}
-H_\nu(\nu)=-H(\nu)\leq& (1-1/b)\liminf_{N\to\infty}\dfrac{1}{N}\log Q_N(L_{N} \in B(\nu, \eta))+C(b-1)\sqrt{\eta}\\
\leq & (1-1/b)\liminf_{N\to\infty}\dfrac{1}{N}\log Q_N(L_{N} \in \sO)+C(b-1)\sqrt{\eta}.
\end{align*}
We now take the limits $\eta\to 0$ and $b\to \infty$ to recover the lower bound in~\eqref{lower:Q}, concluding the proof.

\subsection{Upper bound for compact sets}\label{subsec:upper_bound}
~
\par Let $\sK \subset \M_{1}(\C_{-\tau}^T \times \R^{d})$ be a compact set. We want to prove that
\begin{equation}\label{upper:goal}
\limsup_{N\to\infty} \frac{1}{N} \log Q_{N}( L_{N} \in \sK) \leq -\inf_{\mu \in \sK} H(\mu).
\end{equation} 

\par As in Subsection~\ref{lower}, we will use the upper bound we already have for the sequence $Q_\nu^{\otimes N}(L_N \in \sK).$ To recover a bound for $Q_N(L_N \in \sK)$ we will need to compare the Radon-Nikodym derivatives and the rate functions using Lemmas~\ref{lemma:rate_function} and~\ref{lemma:estimates}. For this, we will cover $\sK$ with closed balls $B[\nu,\eta]$ of small radius and our comparisons will take place on $\sK\cap B[\nu,\eta].$

\par To formalize the previous idea, let $\eta>0$ be arbitrary. Since $\sK$ is compact, it is possible to find $M_\eta \in \N$ and a collection $\{\nu_{i,\eta} \in \sK\,:\,1\leq i \leq M_\eta\}$ such that
\begin{equation*}
\sK \subseteq \bigcup_{i=1}^{M_\eta}B[\nu_{i,\eta}, \eta].
\end{equation*}
In particular,
\begin{equation}
Q_N(L_N \in \sK)\leq \sum_{i=1}^{M_\eta}Q_N(L_N \in \sK\cap B[\nu_{i,\eta},\eta]).
\end{equation}

\par To take the logarithm in the bound above, we will use the following well-known fact. If $(A_{N})_{N \in \N}$ and $(B_{N})_{N \in \N}$ are sequences of positive real numbers, then
\begin{equation*}
\limsup_{N\to \infty}\frac{1}{N} \log (A_{N}+B_{N}) \leq \max \bigg\{ \limsup_{N\to\infty} \frac{1}{N}\log A_{N}, \limsup_{N\to\infty} \frac{1}{N} \log B_{N} \bigg\}.
\end{equation*}
By directly applying this estimate, we obtain
\begin{multline}\label{eq:limsup_bound}
\limsup_{N\to \infty} \frac{1}{N} \log Q_{N}( L_{N} \in \sK) \leq \\
\max_{1\leq i \leq M_\eta} \limsup_{N\to\infty} \frac{1}{N} \log Q_{N}( L_{N} \in \sK \cap B[\nu_{i,\eta},\eta]).
\end{multline}

\par We now proceed to estimate the RHS of~\eqref{eq:limsup_bound}. Using H\"older's Inequality, for any $b \in (1,\infty)$, 
\begin{multline*}
Q_{N}( L_{N} \in \sK \cap B[\nu_{i,\eta},\eta])=\int_{ \{L_{N} \in \sK \cap B[\nu_{i,\eta},\eta]\}}\dfrac{\d Q_N}{\d Q_{\nu_{i,\eta}}^{\otimes N}} \, \d Q_{\nu_{i,\eta}}^{\otimes N}\\
\leq Q_{\nu_{i,\eta}}^{\otimes N}(L_{N} \in \sK \cap B[\nu_{i,\eta},\eta])^{1-1/b}\bigg[\int_{ \{L_{N} \in \sK \cap B[\nu_{i,\eta},\eta]\}}\bigg(\dfrac{\d Q_N}{\d Q_{\nu_{i,\eta}}^{\otimes N}}\bigg)^b \, \d Q_{\nu_{i,\eta}}^{\otimes N}\bigg]^{1/b}.
\end{multline*}

\par Imposing that $\eta \in (0,1)$, Item~\ref{item:RN2nd} of Lemma~\ref{lemma:estimates} (see Equation~\eqref{eq:l52-3}) says that there exists a constant $C=C(\pnorm{\text{SL}}{f},h_*)$ such that
\[\int_{ \{L_{N} \in \sK \cap B[\nu_{i,\eta},\eta]\}}\bigg(\dfrac{\d Q_N}{\d Q_{\nu_{i,\eta}}^{\otimes N}}\bigg)^b \, \d Q_{\nu_{i,\eta}}^{\otimes N}\leq \exp\left\{NC(b^2+b)\sqrt{\eta}\right\}.\]
Therefore,
\begin{multline*}
\dfrac{1}{N}\log Q_{N}( L_{N} \in \sK \cap B[\nu_{i,\eta},\eta])\leq (1-1/b)\dfrac{1}{N}\log Q_{\nu_{i,\eta}}^{\otimes N}(L_{N} \in \sK \cap B[\nu_{i,\eta},\eta])\\
+C(b+1)\sqrt{\eta}.
\end{multline*}
Taking the superior limit as $N\to\infty$ and the bound we already have from the LDP of $Q_{\nu_{i\,\eta}}^{\otimes N}(L_N\in \cdot\,)$ with rate function $H_{\nu_{i,\eta}}$ we obtain
\begin{multline*}
\limsup_{N\to\infty}\dfrac{1}{N}\log Q_{N}( L_{N} \in \sK \cap B[\nu_{i,\eta},\eta])\leq \\
(1-1/b)\big(-\inf\{H_{\nu_{i,\eta}}(\mu)\,:\,\mu \in \sK \cap B[\nu_{i,\eta},\eta]\}\big)+C(b+1)\sqrt{\eta}.
\end{multline*}

\par The inequality above together with~\eqref{eq:limsup_bound} yields
\begin{multline}\label{eq:loc_ap_1}
\limsup_{N\to\infty}\frac{1}{N} \log Q_{N}( L_{N} \in \sK) \leq \\
(1-1/b)\left(- \min_{1\leq i\leq M_\eta}\inf\left\{H_{\nu_{i,\eta}}(\mu)\,:\mu \in \sK \cap B[\nu_{i,\eta},\eta]\right\}\right)\\
+C(b+1)\sqrt{\eta}.
\end{multline}

\par We now examine the limit of the quantity above as $\eta \to 0$. Let
\begin{equation*}
F(\eta) = \min_{1\leq i\leq M_\eta}\inf\left\{H_{\nu_{i,\eta}}(\mu)\,:\mu \in \sK \cap B[\nu_{i,\eta},\eta]\right\}.
\end{equation*}
Our goal is to prove that
\begin{equation}\label{eq:limit_F}
\lim_{\eta \to 0} F(\eta) = \inf_{\mu \in \sK}H(\mu).
\end{equation}

Once this is done, by taking the limit $\eta \to 0$ in~\eqref{eq:loc_ap_1}, we obtain
\begin{equation}
\limsup_{N\to\infty}\frac{1}{N} \log Q_{N}( L_{N} \in \sK) \leq (1-1/b)\left(- \inf_{\mu \in \sK} H(\mu)\right).
\end{equation}
Now we let $b \to \infty$ to conclude the proof.

Let us now focus on the proof of~\eqref{eq:limit_F}. First, notice that, in view of Remark~\ref{remark:infinite_entropy}, the above holds trivially if $\inf_{\mu \in \sK}I(\mu|P) = \infty$, since both sides of the equation are infinite.

We now assume that $\inf_{\mu \in \sK}I(\mu|P)$ is finite. In this case, $G:=\inf_{\mu \in \sK}H(\mu)$ is finite and $F(\eta)$ is bounded from above by $2\inf_{\mu \in \sK}I(\mu|P)+C$, due to Item \ref{prop:gamma_finite} of Lemma~\ref{lemma:rate_function}.

For each $L > 2\inf_{\mu \in \sK}I(\mu|P)+C$, define the auxiliary functions
\begin{equation*}
F_{L}(\eta) = \min_{1\leq i\leq M_\eta}\inf\left\{H_{\nu_{i,\eta}}(\mu)\,:\mu \in \sK \cap B[\nu_{i,\eta},\eta] \text{ and } I(\mu|P) \leq L\right\}
\end{equation*}
and the quantity
\begin{equation*}
G_{L} = \inf\left\{H(\mu):\mu \in \sK \text{ and } I(\mu|P) \leq L\right\}.
\end{equation*}
Notice that we have the bounds $F \leq F_{L}$ and $G \leq G_{L}$, for all $L$, since we are taking the infimum over smaller sets. Furthermore, we will prove that there exists $L'$ such that $F_{L'}=F$ and $G=G_{L'}$.

Let us first assume that $F_{L'}=F$ and $G_{L'}=G$, and conclude~\eqref{eq:limit_F}. By applying Item~\ref{lemma:perturbation} of Lemma~\ref{lemma:rate_function}, we immediately obtain
\begin{equation*}
\begin{split}
F(\eta) & = F_{L'}(\eta) = \min_{1\leq i\leq M_\eta}\inf\left\{H_{\nu_{i,\eta}}(\mu)\,:\mu \in \sK \cap B[\nu_{i,\eta},\eta] \text{ and } I(\mu|P) \leq L' \right\} \\
& \leq \min_{1\leq i\leq M_\eta}\inf\left\{H(\mu)+c(L'+1)\eta\,:\mu \in \sK \cap B[\nu_{i,\eta},\eta] \text{ and } I(\mu|P) \leq L'\right\} \\
& = \inf\left\{H(\mu):\mu \in \sK \text{ and } I(\mu|P) \leq L'\right\}+c(L'+1)\eta \\
& = G_{L'}+c(L'+1)\eta \\
& = \inf_{\mu \in \sK}H(\mu)+c(L'+1)\eta.
\end{split}
\end{equation*}
In an analogous way, one obtains
\begin{equation}
F(\eta) \geq \inf_{\mu \in \sK}H(\mu)-c(L'+1)\eta.
\end{equation}
From these two estimates, we deduce that
\begin{equation*}
\lim_{\eta \to 0} F(\eta) = \inf_{\mu \in \sK}H(\mu).
\end{equation*}

To conclude, we only have to prove that $F=F_{L'}$ and $G=G_{L'}$, for some $L'$ large enough. We focus on $F_{L}$, the proof of the statement for $G_{L}$ follows the same lines. For each $\eta \in (0,1]$, whenever $\mu \in \sK \cap B[\nu_{i,\eta}, \eta]$  is such that
\begin{equation*}
H_{\nu_{i,\eta}}(\mu) \leq F(\eta)+1,
\end{equation*}
Estimate~\eqref{eq:bound_gamma} implies that
\begin{equation*}
(1-\delta)I(\mu|P)-c \leq H_{\nu_{i,\eta}}(\mu) \leq F(\eta)+1 \leq 2\inf_{\tilde{\mu} \in \sK}I(\tilde{\mu}|P)+C+1,
\end{equation*}
which gives
\begin{equation*}
I(\mu|P) \leq \frac{1}{1-\delta}\left(2\inf_{\tilde\mu \in \sK} I(\tilde\mu|P)+C+1+c\right).
\end{equation*}
This implies $F=F_{L'}$, for any
\begin{equation*}
L' \geq \frac{1}{1-\delta}\left(2\inf_{\tilde\mu \in \sK} I(\tilde\mu|P)+C+1+c\right),
\end{equation*}
and concludes the proof.

\subsection{Exponential tightness}
\label{sec:exptight}
~
\par The goal of this subsection is to prove exponential tightness for the sequence $Q_N(L_N \in \cdot\,)$, that is, we want to prove that, for any $M>0$, there exists a compact set  $\sK_M\subset\sM_1\big(\sC_{-\tau}^{T} \times \R^{d}\big)$ such that
\begin{equation}\label{goal:exptight}
\limsup_{N\to\infty}\frac{1}{N} \log Q_N(L_{N} \notin \sK_{M})\leq -M.
\end{equation}

The proof will rely on the fact that $P^{\otimes N}(L_N\in \cdot\,)$ is exponentially tight and that we can relate this sequence with $Q_N(L_N \in \cdot\,)$ via Radon-Nikodym derivatives and Lemma~\ref{lemma:estimates}.

Let $M>0$ be fixed. From~\cite[Exercise 1.2.19]{dembo2009large}, there exists a compact $\sK_M\subset\sM_{1}\left(\sC_{-\tau}^{T} \times \R^{d}\right)$ such that
\begin{equation}\label{P:exptight}
\limsup_{N\to\infty}\frac{1}{N} \log P^{\otimes N}(L_{N} \notin \sK_{M})\leq -M.
\end{equation}

On other hand,
\begin{align*}
Q_N(L_{N} \notin \sK_{M})=&\int_{\{L_{N} \notin \sK_{M}\}}
\dfrac{\d Q_N}{\d P^{\otimes N}}(\vec{x}_T,\vec{\omega}) \, \d P^{\otimes N}(\vec{x}_T,\vec{\omega}).
\end{align*}
By Cauchy-Schwartz Inequality and Item~\ref{estimate:exptight} of Lemma~\ref{lemma:estimates} (see Equation~\eqref{eq:l52-3}), there exists a constant $C=C(h_*,\supnorm{f},T)$ such that
\begin{align*}	
Q_{N}(L_{N} \notin \sK_{M}) & \leq P^{\otimes N}(L_{N} \notin \sK_{M})^{\frac{1}{2}}\bigg[\int \bigg(\dfrac{\d Q_N}{\d P^{\otimes N}}\bigg)^2\, \d P^{\otimes N}\bigg]^{\frac{1}{2}}\\
&\leq P^{\otimes N}(L_{N} \notin \sK_{M})^{\frac{1}{2}}\exp\Big\{NC(4-2)\Big\}^{1/2}.
\end{align*}
As an immediate consequence of~\eqref{P:exptight},
\begin{equation*}
\limsup_{N\to \infty}\frac{1}{N}\log Q_{N}(L_{N} \notin \sK_{M}) \leq -\frac{M}{2}+C,
\end{equation*}
which is enough to conclude, since the last bound goes to $-\infty$ as $M\to\infty.$

\section{The minimizer of $H$}\label{sec:minimizer}
~
\par In this section, we prove Theorem~\ref{thm:minimizer}. Existence and uniqueness are proved in the next subsection. We verify the relation with the path-dependent McKean-Vlasov PDE in Subsection \ref{sec:pde}.

\subsection{Existence and uniqueness}
~
\par In this section we prove that $H$ has a unique minimizer. The main idea is that we can translate the problem of finding minimizers of $H$ to the problem of finding fixed points of the map $\mu\to Q_\mu$.

\par Since $H$ is a good rate function, it admits at least one minimizer $\mu$. Applying Definition~\ref{def:LDP} for the full space $\sM_1(\sC_{-\tau}^T\times \R^d)$, one can easily see that $\inf_\nu H(\nu)=0$. Therefore, any minimizer $\mu$ of $H$ satisfies $H(\mu)=0$. It remains to verify that this minimizer is unique.

\par In view of the characterization provided by Theorem~\ref{t:ldp} and Lemma~\ref{lemma:rate_function}, Item~\ref{lemma:ldp_perturbation}, 
\begin{equation}\label{eq:rememberH}
H(\mu)=H_\mu(\mu)=I(\mu|Q_{\mu}),
\end{equation}
any minimizer $\mu$ of $H$ satisfies $\mu = Q_{\mu}$. If we denote by $\sQ$ the map $\mu\mapsto Q_\mu$ defined in Definition~\ref{def:decoupling}, each minimizer of $H$ corresponds to a fixed point of the map $\sQ$. Our new goal is the following. 
\begin{center}
\textbf{Goal:} The map $\sQ$ has a unique fixed point. 
\end{center}

\par If $\sQ$ were a contraction, we would be done since contraction maps have a unique fixed point. This is not the case, but we will prove that there is $m\geq 1$ such that the composition $\sQ^m$ is a contraction. This concludes the claim, since any fixed point of $\sQ$ is also a fixed point of $\sQ^m$. Therefore we just need to prove the following.
\begin{center}
\textbf{Goal:} For some $m\geq 1$, the map $\sQ^m$ is a contraction. 
\end{center}

\par Assume for a moment that there exists a constant $C$ depending on $f$, $h$ and $T$ such that, for all $t \in [0,T]$,
\begin{equation}\label{eq:gronwall}
\d^t_{\text{BL}}(\sQ(\mu),\sQ(\nu))\leq C\int_0^t\d^s_{\text{BL}}(\mu,\nu)\,\d s.
\end{equation}
Iterating this bound, we obtain, for $n\geq 2$,
\begin{equation}
\begin{split}
\d^{T}_{\text{BL}}(\sQ^n(\mu), \sQ^n(\nu)) & \leq C \int_{0}^{T} \d^{s_{1}}_{\text{BL}}(\sQ^{n-1}(\mu), \sQ^{n-1}(\mu)) \,\d s_{1} \\
& \leq C^{n}\int_{0}^{T}\int_{0}^{s_{1}} \cdots \int_{0}^{s_{n-1}}\d_{\text{BL}}^{s_{n}}(\mu, \nu) \,\d s_{n} \dots \d s_{1}.
\end{split}
\end{equation}
Using the bound $\d_{\text{BL}}^{s_{n}}(\mu, \nu)\leq \d_{\text{BL}}^{T}(\mu, \nu)$, we get
\begin{equation}
\d^{T}_{\text{BL}}(\sQ^n(\mu), \sQ^n(\nu))) \leq \frac{(CT)^{n}}{n!}d^{T}_{\text{BL}}(\mu, \nu).
\end{equation}
Choosing $m$ such that $(CT)^{m}/m!<1$, we obtain that the map $\sQ^{m}$ is a contraction and conclude the proof.

\par We finish this section verifying~\eqref{eq:gronwall}, concluding the proof of the first part of Theorem~\ref{thm:minimizer}.

\par In order to prove~\eqref{eq:gronwall}, fix two measures $\mu,\nu\in \sM_1(\sC_{-\tau}^T\times \R^d).$ Construct $\psi^{\omega,\mu}$ and $\psi^{\omega,\nu}$ on the same probability space, by coupling the following variables to be equal: the Brownian motions $B$, the initial conditions $\xi$, and the media variables $\omega$. Under this construction, the laws of $\psi^{\omega,\mu}$ and $\psi^{\omega,\nu}$ are $\sQ(\mu)$ and $\sQ(\nu),$ respectively, and as consequence of \eqref{eq:boundonBL}
\begin{equation}
\d^t_{\text{BL}}(\sQ(\mu),\sQ(\nu))\leq \Ex{\pnorm{[-\tau,t]}{\psi^{\omega,\mu}_t-\psi^{\omega,\nu}_t}}.
\end{equation}
To estimate the RHS of the previous inequality, recall from~\eqref{eq:intdiff-nu} that, almost surely, for all $r \in [0, t]$,
\begin{equation*}
\psi^{\omega,\mu}(r)=\xi(0)+\int_0^r f\big(s, \psi^{\omega, \nu}, \mu, \omega\big)\,\d s+ h\big(\omega\big)B(r)
\end{equation*}
and a similar expression holds for $\psi^{\omega,\nu}$ with the same Brownian motion, the same initial condition, and the same media variables. Subtracting the expressions, we obtain, almost surely, for all $r \in [0,t]$,
\begin{equation*}
\psi^{\omega,\mu}(r)-\psi^{\omega,\nu}(r)=\int_0^r\left[f\big(s, \psi^{\omega, \mu}, \mu, \omega\big)-f\big(s, \psi^{\omega, \nu}, \nu, \omega\big)\right]\,\d s.
\end{equation*}
From the assumptions on $f$ (cf. Section~\ref{subsec:assumptions}), we immediately have, for almost all realizations,
\begin{equation*}
|\psi^{\omega,\mu}(r)-\psi^{\omega,\nu}(r)|\leq \pnorm{\text{SL}}{f}\left(\int_0^r\pnorm{[-\tau,s]}{\psi^{\omega,\mu}_s-\psi^{\omega,\nu}_s}\,\d s + \int_0^r\d^s_{\text{BL}}(\mu,\nu)\,\d s\right).
\end{equation*}
Recall that $\psi^{\omega,\mu}(r)=\psi^{\omega,\nu}(r)$, for all $r \in [-\tau,0]$. Taking the supremum in $r \in [-\tau,t]$, one obtains
\begin{equation*}
\pnorm{[-\tau,t]}{\psi^{\omega,\mu}_t-\psi^{\omega,\nu}_t}\leq \pnorm{\text{SL}}{f}\left(\int_0^t\pnorm{[-\tau,s]}{\psi^{\omega,\mu}_s-\psi^{\omega,\nu}_s}\,\d s + \int_0^t\d^s_{\text{BL}}(\mu,\nu)\,\d s\right).
\end{equation*}
Gronwall's inequality says that if $u(t)$ and $\alpha(t)$ are continuous with $\alpha$ non-decreasing, then
\begin{equation}
u(t)\leq \alpha(t)+\beta\int_0^tu(s)\,\d s \implies u(t)\leq e^{\beta t}\alpha(t).
\end{equation}
From this, we obtain 
\begin{equation*}
\pnorm{[-\tau,t]}{\psi^{\omega,\mu}_t-\psi^{\omega,\nu}_t}\leq e^{\pnorm{\text{SL}}{f} T}\pnorm{\text{SL}}{f}\int_0^t\d^s_{\text{BL}}(\mu,\nu)\,\d s,
\end{equation*}
holds almost surely. Taking expectation, we conclude the proof of~\eqref{eq:gronwall}.

\subsection{Solving the PDE}\label{sec:pde}
~
\par We dedicate this subsection to complement Theorem \ref{thm:minimizer}, showing that $\nu^{*,\omega}$ satisfies a family of coupled path-dependent McKean-Vlasov PDEs. The idea is to use It\^o's Formula for a test function to see that $\nu^{*,\omega}$ satisfies a PDE in the weak sense.

Recall first that $\nu^*=Q_{\nu^*}$ is the annealed law of $V^{\omega}$ that satisfies
\begin{multline*}
\begin{cases} \d V^{\omega}(t) =f\big(t, V^{\omega}, \nu^*, \omega\big)\,\d t+ h\big(\omega\big)\,\d B(t), & 0 \leq t \leq T, \\
V^{\omega}_{0}=\xi_{0}.
\end{cases}
\end{multline*}
\ignore{ This implies that $\nu^{*,\omega}\sim V^{\omega}$ is absolutely continuous with respect to $P^\omega$ which is the law of a Brownian motion reparametrized by time. Therefore, the law of the marginal at time $t\in [0,T]$ satisfies $\nu^{*,\omega}(t)(\d x)=\rho^\omega(t,x)\d x$, that is, $\nu^{*,\omega}(t)$ has a density  with respect to the Lebesgue measure for almost all $\omega$.}

By It\^o's Formula \cite[Theorem~3.3.3]{karatzas}, for any function $\phi:\R\to\R$ with bounded continuous derivatives up order $2$, it holds that
\begin{align*}
\phi(V^\omega(t))-\phi(V^{\omega}(0))=&\int_0^t\phi'(V^{\omega}(s))f(s,V^{*,\omega},\nu^*,\omega) \, \d s\\
&+\int_0^t\phi'(V^{\omega}(s)) \, \d B(s)+\frac{h(\omega)^2}{2}\int_0^t\phi''(V^\omega(s)) \, \d s.
\end{align*}
We write $\nu^{*,\omega}(t)$ for the law of $V^{\omega}(t)$ for $t\in [0,T]$ and by $\nu^{*,\omega}_t$ the law of the path $V^\omega_t\in \sC_{-\tau}^t$. Since $\phi$ and its two derivatives are bounded, we can take expectation with respect to the Brownian motions to obtain that
\begin{align}\label{weak}
\int_{\R}\phi(u)(\nu^{*,\omega}(t)-\nu^{*,\omega}(0))(\d u)=&
\int_0^t\int_{\sC_{-\tau}^T}\phi'(x(s))f(s,x_s,\nu^*,\omega)\nu^{*,\omega}(\d x_T) \, \d s\nonumber\\
&+\frac{h(\omega)^2}{2}\int_0^t\int_\R\phi''(x(s))\nu^{*,\omega}(\d x_T) \, \d s.
\end{align}

With this, we define the operator $L_{\nu^{*,\omega}}$ via
\[L_{\nu^{*,\omega}}(\phi)(t,x_T)=f(t,x_t,\nu^*,\omega)\phi'(x(t))+\frac{h(\omega)^2}{2}\phi''(x(t))\]
to rewrite \eqref{weak} as
\begin{align*}
\int_{\R}\phi(u)(\nu^{*,\omega}(t)-\nu^{*,\omega}(0))(\d u)=&
\int_0^t\int_{\sC_{-\tau}^T}L_{\nu^{*,\omega}}(\phi)(s,x_T)\nu^{*,\omega}(\d x_T) \, \d s.
\end{align*}

\appendix

\section{The Radon-Nikodym derivative}\label{app:radon-nikodym}
~
\par Here we prove the claims about the Radon-Nikodym derivative of the solution of the system with respect to $P$ (see Subsection~\ref{sec:heuristics}).

\par Recall from~\eqref{def:D} that
\begin{equation*}
D^{\nu,\omega}(y_{T}):=\frac{1}{h(\omega)^{2}}\int_{0}^{T}f(t,y_{t}, \nu, \omega)\,\d y(t)- \frac{1}{2h(\omega)^{2}}\int_{0}^{T}f(t, y_{t}, \nu, \omega)^{2} \,\d t.
\end{equation*}

\par For each fixed vector $\vec{\omega} \in \R^{Nd}$ and collection of initial conditions $(\xi_{0}^{i})_{i \in [N]}$, define
\begin{multline*}
Z(t) = \exp \left\{ -\sum_{i=1}^{N}\frac{1}{h(\omega^{i})}  \int_0^t f(s, \theta^{i,\omega}, L_N, \omega^{i}) \, \d B^{i}(s) \right. \\
 \left.-\sum_{i=1}^{N} \frac{1}{2 h(\omega^{i})^2} \int_0^t f^2(s, \theta^{i, \omega}, L_N, \omega^{i}) \, \d s \right\}.
\end{multline*}
Since $f$ is bounded, Novikov's Condition implies that the process above is a martingale.

\par If we define
\begin{equation*}
\d \hat{P}:=Z(T) \, \d \mathcal{W},
\end{equation*}
then, by Girsanov's Theorem, $\left(\dfrac{\theta^{i,\omega}(t)-\theta^{i, \omega}(0)}{h(\omega^{i})}\right)_{i \in [N]}$ is an $N$-dimensional Brownian motion with respect to $\hat{P}$. Since $\theta$ satisfies the system of SDEs in Equation~\ref{eq:intdiff}, we can rewrite
\begin{multline*}
Z(t) = \exp \left\{ -\sum_{i=1}^{N}\frac{1}{h(\omega^{i})^{2}}  \int_0^t f(s, \theta^{i,\omega}, L_N, \omega^{i}) \, \d \theta^{i, \omega}(s) \right. \\
 \left.+\sum_{i=1}^{N} \frac{1}{2 h(\omega^{i})^2} \int_0^t f^2(s, \theta^{i, \omega}, L_N, \omega^{i}) \, \d s \right\}.
\end{multline*}

Let $A$ be a measurable subset of $\sC_0^{T}$. By definition,
\begin{equation} \label{eq:P-omega-esperanca}
Q^{\omega}(A) = \int \mathcal{W}(\theta^{\omega} \in A) \,\d \mu_{0}^{\otimes N} = \int E_\mathcal{W}[1_A(\theta^{\omega})] \, \d \mu_{0}^{\otimes N},
\end{equation}
where the integration above is with respect to the initial condition.

Using the Radon-Nikodym derivative of Equation~\ref{eq:P-omega-esperanca} we get
\begin{align*}
E_\mathcal{W}[1_A(\theta^{\omega})] = E_{\hat{P}} [1_A(\theta^{\omega}) Z_T^{-1}] & \\
= E_{\hat{P}} \Bigg[ 1_A(\theta^{\omega})  \exp & \left\{ -\sum_{i=1}^{N}\frac{1}{h(\omega^{i})^{2}}  \int_0^t f(s, \theta^{i,\omega}, L_N, \omega^{i}) \, \d \theta^{i, \omega}(s) \right. \\
& \left.\left.+\sum_{i=1}^{N} \frac{1}{2 h(\omega^{i})^2} \int_0^t f^2(s, \theta^{i, \omega}, L_N, \omega^{i}) \, \d s \right\}\right].
\end{align*}

Since $\left(\dfrac{\theta^{i,\omega}(t)-\theta^{i, \omega}(0)}{h(\omega^{i})}\right)_{i \in [N]}$ is a Brownian motion with respect to $\hat{P}$, averaging $\hat{P}$ with respect to $\mu_{0}^{\otimes N}$ gives the law $\prod_{i=1}^{N} P^{\omega^{i}}$. In particular, integrating with respect to $\mu_{\rm med}^{\otimes N}$, we obtain
\begin{equation*}
\frac{\d Q_N}{\d P^{\otimes N}}\big(\vec{x}_{T}, \vec{\omega}\big)=\exp\Big(N\int D^{L_N,\omega}\big(y_{T}\big) \, L_N( \d y_{T}, \d \omega) \Big),
\end{equation*}
and conclude~\eqref{eq:dQN/dP}. The verification of~\eqref{eq:dQnu/dP} is similar and we omit it here.

\bibliographystyle{plain}
\bibliography{bibliography}

\begin{thebibliography}{10}

\bibitem{bag}
G\'{e}rard Ben~Arous and Alice Guionnet.
\newblock Large deviations for {L}angevin spin glass dynamics.
\newblock {\em Probability Theory and Related Fields}, 102(4):455--509, 1995.

\bibitem{bogachev}
Vladimir~I Bogachev.
\newblock {\em Measure theory}, volume~2.
\newblock Springer Science \& Business Media, 2007.

\bibitem{bdf}
Amarjit Budhiraja, Paul Dupuis, and Markus Fischer.
\newblock Large deviation properties of weakly interacting processes via weak
  convergence methods.
\newblock {\em The Annals of Probability}, 40(1):74--102, 2012.

\bibitem{cabana_touboul_2018}
Tanguy Cabana and Jonathan~D. Touboul.
\newblock Large deviations for randomly connected neural networks: I. spatially
  extended systems.
\newblock {\em Advances in Applied Probability}, 50(3):944--982, 2018.

\bibitem{Pra1996}
Paolo Dai~Pra and Frank den Hollander.
\newblock {M}c{K}ean-{V}lasov limit for interacting random processes in random
  media.
\newblock {\em Journal of Statistical Physics}, 84(3-4):735--772, 1996.

\bibitem{dg}
Donald~A. Dawson and J{\"u}rgen G{\"a}rtner.
\newblock Large deviations from the {M}c{K}ean-{V}lasov limit for weakly
  interacting diffusions.
\newblock {\em Stochastics: An International Journal of Probability and
  Stochastic Processes}, 20(4):247--308, 1987.

\bibitem{Delattre2016}
Sylvain Delattre, Giambattista Giacomin, and Eric Lu{\c{c}}on.
\newblock A note on dynamical models on random graphs and {F}okker--{P}lanck
  equations.
\newblock {\em Journal of Statistical Physics}, 165(4):785--798, 2016.

\bibitem{dembo2009large}
Amir Dembo and Ofer Zeitouni.
\newblock {\em Large deviations techniques and applications}.
\newblock Stochastic Modelling and Applied Probability. Springer Berlin
  Heidelberg, 2009.

\bibitem{GL}
Antonio Galves and Eva Löcherbach.
\newblock Infinite systems of interacting chains with memory of variable length
  — {A} stochastic model for biological neural nets.
\newblock {\em Journal Statistycal Physics}, 151(5):896--921, 2013.

\bibitem{guionnet}
Alice Guionnet.
\newblock Averaged and quenched propagation of chaos for spin glass dynamics.
\newblock {\em Probability Theory and Related Fields}, 109(2):183--215, 1997.

\bibitem{huang}
Xing Huang.
\newblock Path-distribution dependent sdes with singular coefficients.
\newblock {\em arXiv preprint arXiv:1902.08953}, 2019.

\bibitem{karatzas}
Ioannis Karatzas and Steven~E. Shreve.
\newblock {\em Brownian motion and stochastic calculus}.
\newblock Springer, 1998.

\bibitem{kipnis_landim}
Claude Kipnis and Claudio Landim.
\newblock {\em Scaling limits of interacting particle systems}, volume 320.
\newblock Springer Science \& Business Media, 2013.

\bibitem{Kavita2019}
Daniel Lacker, Kavita Ramanan, and Ruoyu Wu.
\newblock Large sparse networks of interacting diffusions.
\newblock {\em arXiv preprint arXiv:1904.02585}, 2019.

\bibitem{lucon}
Eric Lu{\c{c}}on.
\newblock Quenched large deviations for interacting diffusions in random media.
\newblock {\em Journal of Statistical Physics}, 166(6):1405--1440, 2017.

\bibitem{lucon2014}
Eric Lu{\c{c}}on and Wilhelm Stannat.
\newblock Mean field limit for disordered diffusions with singular
  interactions.
\newblock {\em The Annals of Applied Probability}, 24(5):1946--1993, 2014.

\bibitem{lucon2016}
Eric Lu{\c{c}}on and Wilhelm Stannat.
\newblock Transition from {G}aussian to non-{G}aussian fluctuations for
  mean-field diffusions in spatial interaction.
\newblock {\em The Annals of Applied Probability}, 26(6):3840--3909, 2016.

\bibitem{mssz}
Sima Mehri, Michael Scheutzow, Wilhelm Stannat, and Bian~Z. Zangeneh.
\newblock Propagation of chaos for stochastic spatially structured neuronal
  networks with delay driven by jump diffusions.
\newblock {\em The Annals of Applied Probability}, 30(1):175--207, 2020.

\bibitem{muller}
Patrick~E. M{\"u}ller.
\newblock Path large deviations for interacting diffusions with local
  mean-field interactions in random environment.
\newblock {\em Electronic Journal of Probability}, 22, 2017.

\bibitem{Reis2018}
Roberto~I. Oliveira and Guilherme~H. Reis.
\newblock Interacting diffusions on random graphs with diverging average
  degrees: Hydrodynamics and large deviations.
\newblock {\em Journal of Statistical Physics}, 176(5):1057--1087, 2019.

\bibitem{OliveiraReisStolerman2018}
Roberto~I. Oliveira, Guilherme~H. Reis, and Lucas~M. Stolerman.
\newblock Interacting diffusions on sparse graphs: hydrodynamics from local
  weak limits.
\newblock {\em Electronic Journal of Probability}, 25, 2020.

\bibitem{revuz_yor}
Daniel Revuz and Marc Yor.
\newblock {\em Continuous martingales and Brownian motion}, volume 293.
\newblock Springer Science \& Business Media, 2013.

\end{thebibliography}

\end{document}